\documentclass{article}
\usepackage[utf8]{inputenc}
\usepackage{amsmath, amsthm, amssymb, mathpazo, isomath, mathtools}
\usepackage{subcaption,graphicx,pgfplots}
\usepackage{fullpage}
\usepackage{booktabs}
\usepackage{hyperref}
\usepackage{algorithm}
\usepackage{algorithmic}

\title{Modeling of cardiac fibers as oriented liquid crystals}
\author{Nicol\'as A Barnafi, Axel Osses}
\date{}

\newcommand{\Base}{{\texttt{base}}}
\newcommand{\Endo}{{\texttt{endo}}}
\newcommand{\Epi}{{\texttt{epi}}}
\newcommand{\Trans}{{\texttt{trans}}}
\newcommand{\AB}{{\texttt{ab}}}

\renewcommand{\vec}{\vectorsym}
\newcommand{\mat}{\matrixsym}
\newcommand{\ten}{\tensorsym}
\DeclareMathOperator{\grad}{\nabla}
\DeclareMathOperator{\dive}{\text{div}}
\DeclareMathOperator{\curl}{\text{curl}}
\newtheorem{remark}{Remark}
\newtheorem{definition}{Definition}
\newcommand{\sinb}[1]{\sin{(#1)}}
\newcommand{\cosb}[1]{\cos{(#1)}}
\newcommand{\sinsq}[1]{\sin^2{(#1)}}

\newcommand{\da}{[\vec d_\Trans]}
\newcommand{\db}{[\vec d_\AB]}

\newcommand{\tin}{\text{in}}
\newcommand{\ton}{\text{on}}

\newtheorem{lemma}{Lemma}

\usepackage{listings}
\usepackage{xcolor}
\definecolor{codegreen}{rgb}{0,0.6,0}
\definecolor{codegray}{rgb}{0.5,0.5,0.5}
\definecolor{codepurple}{rgb}{0.58,0,0.82}
\definecolor{backcolour}{rgb}{0.95,0.95,0.92}
\lstdefinestyle{mystyle}{
  backgroundcolor=\color{backcolour}, commentstyle=\color{codegreen},
  keywordstyle=\color{magenta},
  numberstyle=\tiny\color{codegray},
  stringstyle=\color{codepurple},
  basicstyle=\ttfamily\footnotesize,
  breakatwhitespace=false,         
  captionpos=b,                    
  keepspaces=true,                 
  numbers=none,                    
  numbersep=5pt,                  
  showspaces=false,                
  showstringspaces=false,
  showtabs=false,                  
  tabsize=2,
  framerule=1.5pt,
  rulecolor=\color{red!60!black}
}
\begin{document}

\maketitle

\begin{abstract}
In this work we propose a mathematical model that describes the orientation of ventricular cardiac fibers. These fibers are commonly computed as the normalized gradient of certain harmonic potentials, so our work consisted in finding the equations that such a vector field satisfies, considering the unitary norm constraint. The resulting equations belong to the Frank-Oseen theory of nematic liquid crystals, which yield a bulk of mathematical properties to the cardiac fibers, such as the characterization of singularities. The numerical methods available in literature are computationally expensive and not sufficiently robust for the complex geometries obtained from the human heart, so we also propose a preconditioned projected gradient descent scheme that circumvents these difficulties in the tested scenarios. The resulting model further confirms recent experimental observations of liquid crystal behavior of soft tissue, and provides an accurate mathematical description of such behavior. 
\end{abstract}
\section{Introduction}
The heart is the blood pump of the body, and it works throughout our entire lifespan. Surgical procedures involving the heart are very invasive because it resides within the rib cage, so it is fundamental to find alternative ways to diagnose, plan and treat it. The progress in computational resources has enabled the use of increasingly complex mathematical models to describe a patient-specific heart (see \cite{AShortHistoryNieder2019} for a review), which encompasses complex mechanisms at many different scales: ion dynamics at cell membranes, electric potential propagation within the myocardium (cardiac muscle tissue), muscle contraction and then blood circulation thanks to the synchronized action of all of these elements \cite{tortora2018principles}. One fundamental physiological property of the heart is that the myocardium is anisotropic \cite{EffectOfTissuRobert1982}, with the anisotropy given by an ordered set of fibers that impact many of the heart's mechanisms, such as deformation and electric conductivity. An adequate description of cardiac fibers is thus at the core of any mathematical model of the heart.

The fibers were initially computed by means of an imaging technique known as Diffusion Tensor Magnetic Resonance Imaging (DT-MRI), which is very expensive and difficult for \emph{in-vivo} patients \cite{HeartMuscleFiZhukovNone}. To alleviate these costs, mathematical models known as Rule Based Models (RBM) were proposed \cite{FiberOrientatiStreet1969}. They are in good agreement with measurements of fiber orientation in the ventricles \cite{ANovelRuleBaBayer2012}, whereas atrial RBMs are still an active area of research \cite{ModelingAtrialKruege2011,ModelingCardiaPiersa2021,AnAutomatePipZheng2021}. This has motivated other areas of research for atrial fiber generation, mainly through data-assimilation \cite{ATechniqueForRoney2019,PhysicsInformeRuizH2022}.

RBMs consist in the manipulation of a series of harmonic potentials with different boundary conditions, i.e. solutions of the Laplace equation, whose gradients are normalized and then combined with histologically observed rotations to obtain an orthonormal basis oriented along the fibers. Despite the differences of the existing approaches \cite{ThermodynamicalRossi2014,GeneratingFibrWong2014,ANovelRuleBaBayer2012,ARuleBasedMeDoste2019}, they follow similar steps as shown in \cite{ModelingCardiaPiersa2021}:

\begin{enumerate}
    \item Provide adequate labels for the mesh geometry: endocardium (inner wall), epicardium (external wall), base (top cut), and the apex (bottom point).
    \item Compute a distance from the endocardium to the epicardium, known as transmural distance. Its gradient yields the transmural direction.
    \item Compute a direction going form the apex to the base, orthogonal to the transmural direction, known as the apicobasal direction.
    \item Define a local coordinate system by complementing the transmural and apicobasal directions with a longitudinal (or transversal) direction. 
    \item Rotate the computed reference frame to match histological observations and finally yield a fiber direction $\vec f$, a cross-fiber direction $\vec n$ and a sheet direction $\vec s$.
\end{enumerate}

Even though these steps are clear and well-defined, there are some limitations hidden within them that are intrinsic to all RBM formulations. These are:

\begin{itemize}
    \item The fibers are computed as the gradients of harmonic functions. Numerically, this means that the resulting fiber field could be discontinuous or inaccurate. This is most evident in thin muscle walls such as the right ventricle and the atria.
    \item There is no known system of equations that the fibers satisfy. This makes it difficult to mathematically analyse the qualitative behavior of fiber orientation. 
    \item The concept of unit-vector interpolation is fundamental. In practice, two different rotations of the longitudinal direction are required on the endocardium and epicardium, so that different boundary rotations, considered correct, are then interpolated within the tissue using a weight such as the transmural distance. The situation is more difficult with the biventricle scenario, because a choice has to be made regarding the combination of the fiber fields in the intra-ventricular septum (IVS). This problem was circumvented in \cite{ANovelRuleBaBayer2012} by interpreting the orientation as a quaternion and using well-established techniques from computer vision for unit-quaternion interpolation \cite{AnimatingRotatShoema1985} to obtain an inter-ventricular interpolation.
\end{itemize}

The third point is the one that has received most attention, whereas the first two are largely unaddressed. Our goal is to tackle these three problems together, which we do by computing a Partial Differential Equation (PDE) satisfied by the fibers. The resulting model is a particular case of a nematic liquid crystal described by the Frank-Oseen theory (see \cite{LiquidCrystalsBall2017} for further references), given by the minimization of the $H_0^1$ norm of the vector field, subject to having unitary norm throughout the domain. This is also known in the literature as harmonic maps. The connection of liquid crystals with living tissue, even though we derived it from a purely mathematical approach, has already been observed in literature \cite{LiquidCrystalsHirst2017}, and only very recently for cardiac fibers \cite{TheNematicChiAuriau2022}. This grants a significant validation not only of our model, but also of the RBMs in general. Beyond Frank-Oseen, there is the Ericksen model that allows for better approximation of singularities by means of an additional variable, and while both the Frank-Oseen and Ericksen theories consider only uniaxial liquid crystals, more general scenarios are better described within the Landau-De Gennes framework. See \cite{LandauDeGenneMajumd2010,LiquidCrystalsBall2017} for further details on these topics.

The existence of solutions to the Frank-Oseen equations and their regularity is well-known for most cases \cite{NonlinearTheorLinF1989}. Additionally, work has been devoted to the numerical approximation of these equations, mainly through the study of the saddle point problem arising from the first order conditions \cite{ASaddlePointHuQi2009,ConstrainedOptAdler2016,BlockPreconditBeik2018,AugmentedLagraXiaJ2021}.  In liquid crystal theory, the geometries considered so far are simple. This means that the existing numerical solvers have not been tested for robustness in complex geometries. In fact, we will show that they fail under such conditions, so we propose a preconditioned projected gradient descent scheme that is robust and optimal in all our tests, even in the presence of singularities. 

This work is structured as follows: In Section \ref{section:model} we derive our proposed PDE for an abstract vector field and highlight its main properties for our application. After motivating the use of the Frank-Oseen model for cardiac fibers, we thoroughly characterize all vector fields used in RBMs as nematic liquid crystals. We conclude this section by showing that many important physical properties can be shown analytically, such as the vector interpolation property, and the arisal of the apex singularity. In Section \ref{section:methodology} we show how a non-standard Dirichlet condition from our model is implemented in practice, and then present our preconditioned projected gradient descent strategy to solve the proposed model. We show numerically that it is optimal and more robust than what has been proposed in the literature. In Section \ref{section:application} we numerically study our approach by doing the following tests: (i) a comparison of the use of our approach as a replacement of a potential based fiber field, (ii) a numerical computation of the possible deviation of the fiber field from being a nematic liquid crystal, (iii) a convergence study where we verify the validity of the Aubin-Nitsche trick for our model, (iv) a test where we compare our approach with the standard one in a slender wall scenario, and (v) a simple contraction test where we study the impact of using our model in a mechanical simulation. In Section \ref{section:singularities} we review the relevant theoretical aspects of topological defects in liquid crystals and how they are be applied to the cardiac fiber context. Finally, in Section \ref{section:discussion} we conclude our work and discuss possible future directions. 

\subsection*{Notations}
Let us consider an open connected and Lipschitz set $\Omega\subset \mathbb R^3$ with boundary $\partial\Omega$, together with the classical Sobolev spaces \cite{Evans2022partial} of functions $f:\Omega\to \mathbb R$ that are square-integrable and that have a square-integrable gradient, denoted  $L^2(\Omega)$ and $H^1(\Omega)$ respectively, with the classical norms $\|\cdot\|_{L^2(\Omega, X)}$ and $\|\cdot\|_{H^1(\Omega, X)}$. We use different fonts for scalars, vectors, matrices, and tensors as $a$, $\vec a$, $\mat A$, and $\ten A$, with their Frobenius norm written as $|\cdot|$. All PDEs are understood weakly as posed in the dual space $\left( H^1(\Omega)\right)'$, or in another suitable space according to the boundary conditions. Dirichlet boundary conditions are defined using the trace operator $\gamma_D:H^1(\Omega)\to H^{1/2}(\partial\Omega)$, and Neumann boundary conditions using the normal derivative trace operator $\gamma_N:\vec H^1(\Omega)\to H^{-1/2}(\partial\Omega)$, given formally by $\Gamma_N \vec u= \grad \vec u\cdot \vec N$, where $\vec N$ stands for the unit outwards normal of the domain $\Omega$. We denote the space of functions satisfying the boundary condition $u=g$ in a subset $\Gamma_D$ of the boundary by  $ H^1_{g}(\Omega) \coloneqq \{u \in H^1(\Omega): \gamma_D u = g \,\text{on}\,\Gamma_D\},$ with $g$ sufficiently regular and $\Gamma_D$ understood from the context to avoid excessive notation. The gradient of a tensor is defined as $(\grad u^{kl})_i = \partial_{x_i} u^{kl}$, and the divergence of a tensor is understood row-wise, i.e. $(\dive \ten A)_i = \sum_j \partial_{x_j} A_{ij}$, with the Laplace operator given by $\Delta \coloneqq \dive \grad$. The outer product between two vectors is given by $\vec a \otimes \vec b= \vec a \vec b^T$. The Gateaux derivative of a functional $\Theta$ at point $\vec u$ in direction $\vec v$ is denoted by $d\Theta(\vec u)[\vec v]$. Finally, Computations will be performed using Einstein's index notation, which establishes that summations are implied by repeated indexes, meaning that the following identities hold: $ \vec a\cdot \vec b = a_i b_i, (\mat A\mat B)_{ij}= A_{ik}B_{kj} $. It is customary to drop the bold symbol of the vector of matrix when referring to its components, i.e. $(\mat A)_{ij}=A_{ij}$.
\section{The fiber model}\label{section:model}

In this section we derive the proposed model by looking at each of the potentials computed in an RBM model and their corresponding vector field. We conclude this section by showing some analytical properties of the model, which are fundamental for its use as an RBM alternative. We will focus only on a left ventricle geometry, as shown in Figure \ref{fig:lv-geometry}. This geometry $\Omega$ has its boundary $\Gamma \coloneqq \partial \Omega$ divided into the endocardium  $\Gamma_\Endo$, the epicardium $\Gamma_\Epi$ and the base $\Gamma_\Base$ such that $\overline{\Gamma} = \overline{\Gamma_\Endo}\cup \overline{\Gamma_\Epi}\cup\overline{\Gamma_\Base}$. We have not considered in this partition the apex point $\Gamma_\texttt{apex}$, defined as the bottom point of the left ventricle, which is fundamental to induce the singularity present on it.

\begin{figure}
    \centering
    \includegraphics[width=0.4\textwidth]{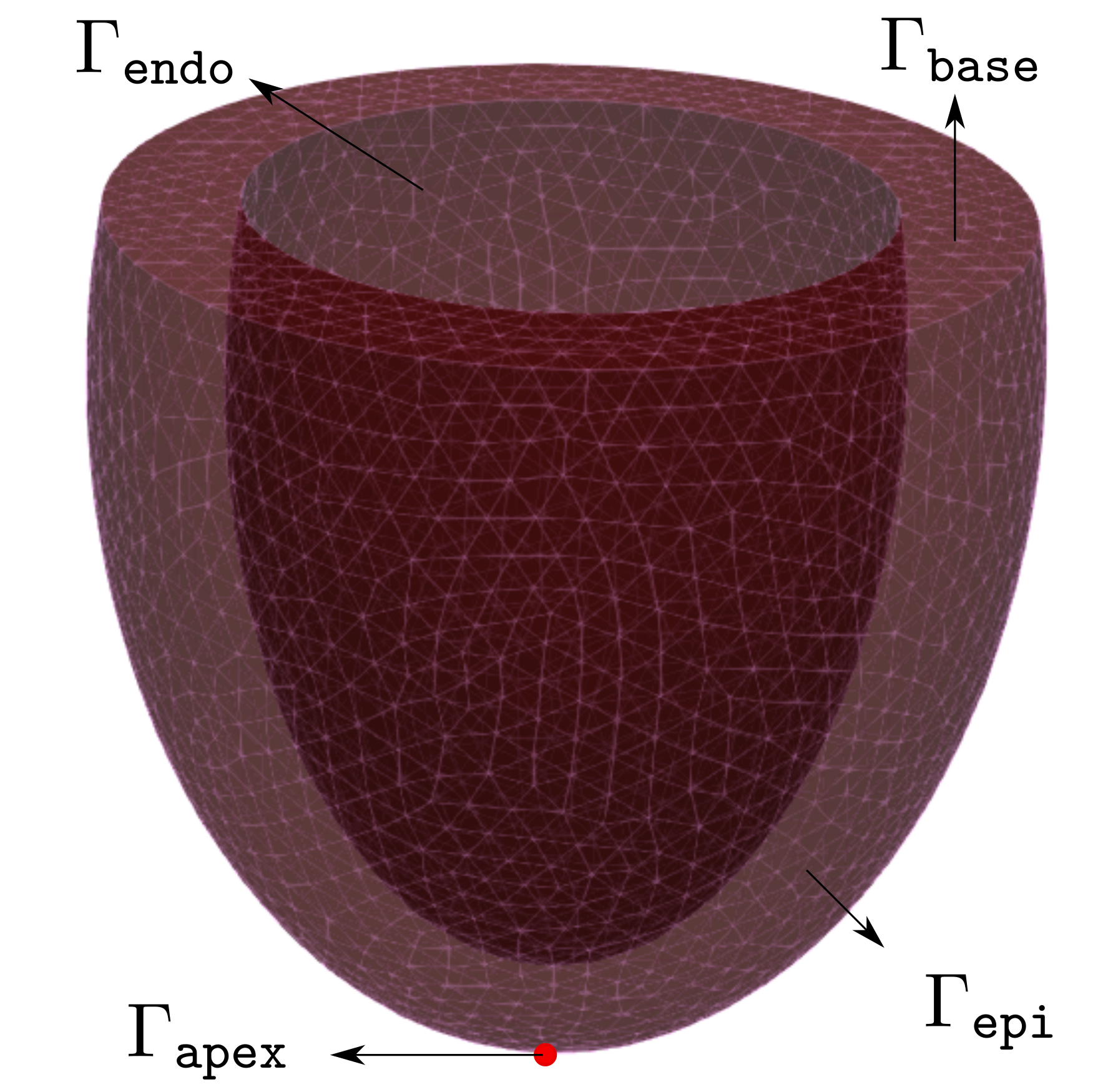}
    \caption{Geometry used for an idealized left ventricle geometry. The boundary tags are given by the endocardium (inner wall), epicardium (outer wall), base (top cut) and the apex (bottom point).}
    \label{fig:lv-geometry}
\end{figure}

\subsection{The transmural and apicobasal vectors}
The main ingredients for computing the fiber vector field are the transmural and apicobasal vectors, both computed as the normalized gradient of harmonic potentials. We represent both problems by the following abstract Poisson problem: Find $\phi$ in $H^1(\Omega)$ such that

\begin{equation}\label{eq:poisson-scalar}
    \begin{aligned}
        -\Delta \phi  &= 0 && \tin \,\Omega, \\
        \gamma_D \phi &= g_D && \ton \, \Gamma_D,\\
        \gamma_N \phi &= 0 && \ton \, \Gamma_N,\\
    \end{aligned}
\end{equation}
where $g_D$ is constant in each connected component of $\Gamma_D$. The steps to compute the transmural and apicobasal vectors are two, given by 

    \begin{itemize}
        \item[]Step 1: Compute the potential $\phi$ by solving \eqref{eq:poisson-scalar} using adequate boundary conditions. For the transmural potential $\phi_\Trans$, this means taking the values 1 on the epicardium and 0 on the endocardium. The apicobasal potential $\phi_\AB$ instead is 0 on the apex and 1 on the base. All remaining boundary conditions are homogeneous Neumann. 
        \item[]Step 2: Compute the normalized gradient as $\vec d^\phi\coloneqq \Pi_{\mathbb S^2}(\grad \phi)$, where $\Pi_{\mathbb S^2}$ stands for the projection into $\mathbb S^2\coloneqq \{\vec y: |\vec y|=1\,\text{ a.e. in $\Omega$}\}$.
    \end{itemize}
\begin{remark}
    A very important aspect of step 2 is that the projection is not defined for $\vec y=\vec 0$. As we will discuss further ahead, one possible solution is to consider a smoothed projection such that $\Pi_{\mathbb S^2}\vec 0 = \vec 0$. The simplest solution in practice is to approximate such projection with 
    $$ \Pi_{\mathbb S^2}^\epsilon \coloneqq \frac{\vec y}{\epsilon + |\vec y|}, $$
    where $\epsilon$ is a small number.
\end{remark}

We highlight that these steps do not have a clear mathematical objective beyond the well-established practical application. To see this, we first note that \eqref{eq:poisson-scalar} is equivalent to solving the minimization problem
    $$ \min_{\varphi \in H_{\Gamma_D}^1(\Omega)} \Psi(\grad \varphi) \coloneqq \frac 1 2 \int_\Omega |\grad \varphi|^2\,dx. $$
In step 2 the gradient is normalized, so that the energy becomes constant:
    $$ \Psi\left(\frac{\grad \varphi}{|\grad \varphi|}\right) = \frac 1 2 \int_\Omega \left|\frac 1 {|\grad \varphi|} \grad \varphi\right|^2\,dx = \frac{|\Omega|}{2}. $$
Still, from \eqref{eq:poisson-scalar} we can take the gradient of the PDE, and using the vector calculus identities $\grad \dive = \dive \grad + \curl \curl$ and $\curl \grad=0$ we obtain that each component of $\grad \phi$ is also harmonic:
    $$ \dive \grad (\grad \phi) = \grad \dive (\grad \phi) - \curl \curl \grad \phi = 0.$$
This means that $\grad \phi$ is the minimizer of the following minimization problem:
    $$ \min_{\vec y\in \vec H_{\Gamma^{\vec d}}^1(\Omega)} \frac 1 2 \int_\Omega |\grad \vec y|^2\,dx, $$
where $\Gamma^{\vec d}$ represents a boundary condition yet to be defined, which we characterize in Section \ref{section:bcs}.  Motivated by the previous computations, we can rewrite the previous steps in terms of $\vec d^\phi$ as:

    \begin{itemize}
        \item[]Step 1': Compute the vector field $\vec d$ that solves the following minimization problem:
            $$ \min_{\vec y\in \vec H_{\Gamma^{\vec d}}^1(\Omega)} \frac 1 2 \int_\Omega |\grad \vec y|^2\,dx.$$
        \item[]Step 2': Compute the normalized vector field as $\vec d^\phi\coloneqq \Pi^\epsilon_{\mathbb S^2}(\vec d)$.
    \end{itemize}
In this formulation, steps 1' and 2' are no longer incompatible. Instead, they can be seen as the first iteration of a projected gradient descent for the following problem:
    
    \begin{equation}\label{eq:fiber-min}
        \min_{\vec d\in \vec H_{\Gamma^{\vec d}}^1(\Omega)\cap \mathbb S^2} \frac 1 2 \int_\Omega |\grad \vec d|^2\,dx.
    \end{equation}
This can be regarded as \emph{the} minimization principle associated to the transmural and apicobasal vector fields, which we have derived simply by completing the already established procedure for computing them. The Euler-Lagrange equations associated to problem \eqref{eq:fiber-min} are given by finding a vector field $\vec d$ in $\vec H_{\Gamma^{\vec d}}^1(\Omega)$ and a Lagrange multiplier $\lambda$ in $L^2(\Omega)$ such that
    
\begin{equation}\label{eq:fiber-saddle}
    \begin{aligned}
        -\Delta \vec d + 2\lambda \vec d &= \vec 0 &&\tin\, \Omega,\\
        \vec d\cdot \vec d \qquad &=1&&\tin\,\Omega.
    \end{aligned}
\end{equation}
We note that the multiplier satisfies $\lambda = -\frac 1 2|\nabla \vec d|^2$ \cite{LiquidCrystalsBall2017}, which can be inferred from the balance equations. Interestingly, this problem coincides with the Frank-Oseen equations for nematic liquid crystals under the \emph{full anchoring} and \emph{one-constant} hypotheses \cite{LiquidCrystalsBall2017}. For the forthcoming analysis, we will make the following definition:

\begin{definition}
We say that a vector field $\vec d$ behaves as a nematic liquid crystal (NLC), or simply is a NLC, if it is a minimizer of \eqref{eq:fiber-min} or if it is a solution of system \eqref{eq:fiber-saddle} with $\lambda = -\frac 1 2 |\grad \vec d|^2$. This definition is independent of the boundary conditions.
\end{definition}

Naturally, the previous definition yields the following lemma: 

\begin{lemma}
The transmural vector $\vec d_\Trans\coloneqq \grad \phi_\Trans$ and the apicobasal vector $\vec d_\AB\coloneqq \grad \phi_\AB$ behave as nematic liquid crystals.
\end{lemma}

In addition, we define a loaded Frank-Oseen model given by
    \begin{equation}\label{eq:fiber-saddle-loaded}
        \begin{aligned}
            -\Delta \vec d + 2\lambda \vec d &= \vec g &&\tin\, \Omega,\\
            \vec d \cdot \vec d &= 1 &&\tin\,\Omega.
        \end{aligned}
    \end{equation}

\begin{definition}
We say that a vector field $\vec d$ behaves as a loaded nematic liquid crystal (LNLC), or simply is a LNLC, if it is a solution of system \eqref{eq:fiber-saddle-loaded}, where $\lambda = \frac 1 2\left(\vec g\cdot \vec d - |\grad \vec d|^2\right)$. This definition is independent of the boundary conditions.
\end{definition}

\begin{remark}
If $ \vec d$ is a loaded nematic liquid crystal, then it is a minimizer of the following problem: 
    $$ \min_{\vec y\in\vec H^1_{\Gamma^{\vec d}(\Omega)\cap \mathbb S^2}} \frac 1 2 \int_\Omega |\grad \vec y|^2\,dx - \int_\Omega \vec g \cdot\vec y\,dx,$$
and replacing the loaded Lagrange multiplier gives the following equation:
    $$ -\Delta \vec d -|\grad \vec d|^2\vec d = (\ten I - \vec d \otimes \vec d)\vec g.$$
It holds in particular that if $\vec g\cdot \vec d=0$, then $\lambda = -\frac 1 2|\grad \vec d|^2$ as in the unloaded case.
\end{remark}
We conclude the presentation of the model by highlighting one additional advantage of considering the gradient of the potential as the primary variable. Typical heart geometries have slender walls in the atria and in the right ventricle, in many cases exhibiting a width of up to one mesh element (tetrahedron or hexahedron). In such scenarios, the gradient of a first order potential is given by a constant in the element, meaning that is will be impossible to depict any rotation through the tissue. Even if the wall presents two elements, such a coarse description of the geometry can lead to a severe lack of accuracy from the point of view of the discretization. In our formulation, if one element is used to describe the tissue walls, we can impose the boundary conditions exactly, so this problem will not be present.

\subsubsection{Transforming the boundary conditions}\label{section:bcs}
In this section we fully characterize the boundary conditions and the boundary $\Gamma^{\vec d}$. For this, we consider separately how to transform the Dirichlet and Neumann boundary conditions from \eqref{eq:poisson-scalar}. 

\paragraph{Neumann boundary conditions.} Neumann boundary conditions are given by 
        $$ \grad \phi \cdot \vec N = \vec d^\phi\cdot \vec N = 0,$$
        meaning that Neumann boundary conditions of the potential are translated into normal Dirichlet boundary conditions in the vector field. We write this condition, and the corresponding boundary as 
        $$ \vec d^\phi\cdot \vec N = 0, \quad\text{on $\Gamma^{\vec d}_{D, \vec N}$}. $$
        This boundary condition must be complemented with the behavior in the tangential direction. For this, we define the tangential projection as $\Pi_{\vec \tau}\coloneqq \ten I - \vec N\otimes \vec N$, which yields the boundary condition
        $$ \Pi_{\vec \tau}[\grad \vec d]\vec N = \vec 0\quad\text{on $\Gamma^{\vec d}_{D, \vec N}$}.$$
\begin{remark}
    One might wonder whether having a homogeneous tangential derivative is the correct choice. Indeed, we have chosen it only for simplicity. 
\end{remark}

\paragraph{Dirichlet boundary conditions.} In this case, we can derivate the boundary conditions to obtain a tangential representation of the gradient. To avoid introducing notation from differential geometry, we only look at one connected component of the Dirichlet boundary condition given by $\phi = g_D$ on $\Gamma_D$, and define the tangential vectors as $\vec \tau_1$ and $\vec \tau_2$. With them, we can formally compute the tangential derivative of the Dirichlet boundary condition as
        $$ \grad_{\vec \tau} \phi = \sum_i(\grad\phi\cdot \vec \tau_i)\vec \tau_i = \sum_i (\grad g_D \cdot \vec \tau_i)\vec \tau_i = \vec 0,$$
where the last equality is true because $g_D$ is constant. Because of this, we conclude that 
        $$ \grad_\tau \phi = \sum_i (\vec d\cdot \vec \tau_i)\vec\tau_i = (\ten I - \vec N\otimes \vec N)\vec d = 0, $$
which combined with the unitary norm constraint yields the following definition:
        $$ \vec d = \vec N, \quad \text{on $\Gamma^{\vec d}_{D, \vec \tau}$}.$$

\paragraph{The apex boundary condition.} The apicobasal function is fundamental for obtaining the apex singularity, which is done by imposing the boundary conditions $ \phi = 0$ on $\Gamma_\texttt{apex}$ and $\phi = 1$ on $\Gamma_\Base$.  The boundary condition on the apex is not theoretically sound, as it is imposed on a single point, but other other automated approaches have proven unsuccessful in patient-specific geometries. We thus transform this condition into another singular condition for the vector problem. Before normalization, we can restrict our analysis to a small ball around the apex $\vec x_0$, where the solution will be given by $ \vec d(\vec x) = \vec x - \vec x_0$.
To normalize this function, we consider a smoothed projector given by 
    $$ \Pi_\epsilon(\vec d)(\vec x) = \begin{cases} \vec d(\vec x)/|\vec d(\vec x)| & \vec x \in [B(\vec x_0, \epsilon)]^c \\ \alpha \vec d(\vec x) & \vec x \in B(\vec x_0, \epsilon) \end{cases}, $$
with $\alpha$ such that the function is continuous and $\epsilon$ sufficiently small. This projector can be well-defined in the discrete setting as well, it suffices to consider $\epsilon$ to be smaller than the smallest element edge on the mesh. Considering $\epsilon\to 0$ gives that the limit projector is given by
    $$ \Pi_\infty(\vec d)(\vec x)= \begin{cases} \vec d(\vec x)/|\vec d(\vec x)| & \vec x \neq \vec x_0 \\  \vec 0 & \vec x = \vec x_0 \end{cases}. $$
This suggests that the boundary condition on apex for the vector problem should be $ \vec d = \vec 0$ on $\Gamma_\texttt{apex}.$

\paragraph{The resulting model.} The resulting minimization problem for both the transmural and apicobasal vector fields is thus defined in the following spaces:
    $$ \vec H_{\Gamma_D^{\vec d}}^1(\Omega) \coloneqq \{\vec y \in \vec H^1(\Omega):\, \vec y = \vec N\quad\text{ on }\Gamma_{D, \vec \tau}^{\vec d}, \,\vec y\cdot\vec N = 0\quad\text{ on }\Gamma_{D,\vec N}^{\vec d}\}, \text{ and }$$
    $$ \vec H_{\Gamma_{D}^{\vec d}, 0}^1(\Omega) \coloneqq \{\vec y \in \vec H^1(\Omega):\, \Pi_{\vec \tau}\vec y = \vec 0\quad\text{ on }\Gamma_{D, \vec \tau}^{\vec d}, \, \vec y\cdot\vec N = 0\quad\text{ on }\Gamma_{D,\vec N}^{\vec d}\}.$$
\begin{remark}
    In the test function space, we have considered the condition $\Pi_{\vec \tau}\vec y=\vec 0$ instead of $\vec y = \vec 0$. This avoids technicalities where the test functions can not be embedded into $\mathbb S^2$.
\end{remark}

The resulting minimization problem is given as follows:

    \begin{equation}\label{eq:fiber-min-bcs}
        \min_{\vec d\in \mathcal V} \frac 1 2 \int_\Omega |\grad d|^2\,dx, 
    \end{equation}
where $\mathcal V$ is defined as
    $$ \mathcal V \coloneqq \vec H_{\Gamma_D^{\vec d}}^1(\Omega)\cap \mathbb S^2. $$
By using it, we can compute the transmural and apicobasal vectors, $\vec d_\Trans$ and $\vec d_\AB$, as minimizers of the Frank-Oseen problem, which completely characterizes their behavior as nematic liquid crystals. We do this numerically in Section \ref{section:application}.

\subsubsection{The transversal vector and the fiber field}
The transversal vector is given by $\vec d = \vec d_\Trans \times \vec d_\AB$.  The fiber field is computed by  means of the transmural distance by interpolating the angles that the fibers have with respect to the transmural direction. For this aim we consider a transmurally varying angle $\alpha(\phi_\Trans) = \phi_\Trans \alpha_\Endo + (1 - \phi_\Trans) \alpha_\Epi$ together with a change of basis and a rotation given by 
    $$ \mat B = [\vec d \,\,\vec d_\AB \,\,\vec d_\Trans], \quad \mat R(\phi_\Trans) = \begin{bmatrix}\cos \alpha(\phi_\Trans) & -\sin \alpha(\phi_\Trans) & 0 \\ \sin \alpha(\phi_\Trans) & \cos \alpha(\phi_\Trans) & 0 \\ 0 & 0 & 1 \end{bmatrix}. $$
The fiber field is ultimately defined as $\vec f(\phi_\Trans) \coloneqq \mat Q(\phi_\Trans) \vec d$, with $\mat Q(\phi_\Trans) = \mat B\mat R(\phi_\Trans)\mat B^T$. One may naturally wonder at this point if either $\vec d$ or $\vec f$ behave as nematic liquid crystals, and the answer is that they behave as loaded nematic liquid crystals. We show this in the following two lemmas.

\begin{lemma}
Consider two orthogonal nematic liquid crystals $\vec d_\Trans$ and $\vec d_\AB$, with Langrange multipliers from \eqref{eq:fiber-saddle} given by $\lambda_\Trans$ and $\lambda_\AB$ respectively. Then, the transversal vector field $\vec d$, defined as the product $\vec d=\vec d_\Trans\times \vec d_\AB$, behaves as a loaded nematic liquid crystal, where the external force is given by
    $$ \vec F = 2[\vec d_\Trans]_{k,j}[\vec d_\AB]_{\ell,j}\epsilon_{k\ell i} - 2\da_{m,j}\da_{n,j}\db_m\db_n. $$
In particular, $\vec F\cdot \vec d=0$.
\end{lemma}
\begin{proof}
First note that as $\vec d_\Trans$ and $\vec d_\AB$ are orthogonal, the following holds:
    $$ \da_{i,j}\db_i + \da_i\db_{i,j}=0.$$
We start by computing the norm $|\grad \vec d|^2$ using that $d_{i,j} =(\da_{m,j}\db_n + \da_{m}\db_{n,j})\epsilon_{mni}$ and the identity $\epsilon_{mni}\epsilon_{pqi} = \delta_{mp}\delta_{nq} - \delta_{mq}\delta_{np}$:
\begin{align*}
    |\grad \vec d|^2 &= d_{i,j} d_{i,j} \\
        &= (\da_{m,j}\db_n+\da_m\db_{n,j})\epsilon_{mni}(\da_{p,j}\db_q+\da_p\db_{q,j})\epsilon_{pqi}\\
        &= (\da_{m,j}\db_n+\da_m\db_{n,j})(\da_{p,j}\db_q+\da_p\db_{q,j})(\delta_{mp}\delta_{nq} - \delta_{mq}\delta_{np})\\ 
        &= (\da_{m,j}\db_n+\da_m\db_{n,j})(\da_{m,j}\db_n+\da_m\db_{n,j} \\
        &\qquad\qquad\qquad - \da_{n,j}\db_m+\da_n\db_{m,j}) \\
        &= \da_{m,j}\db_n(\da_{m,j}\db_n+\da_m\db_{n,j} - \da_{n,j}\db_m+\da_n\db_{m,j}) \\
        &\qquad\qquad  +\da_m\db_{n,j}(\da_{m,j}\db_n+\da_m\db_{n,j} - \da_{n,j}\db_m+\da_n\db_{m,j}) \\
        &= \da_{m,j}\da_{m,j} - \da_{m,j}\db_n\da_{n,j}\db_m \\
        &\qquad\qquad + \db_{n,j}\db_{n,j} - \da_m\db_{n,j}\da_n\db_{m,j} \\
        &= |\grad \da|^2 + |\grad \db|^2 - 2\da_{m,j}\da_{n,j}\db_m\db_n, 
\end{align*}
which implies that 
    $$ -\frac 1 2|\grad \vec d|^2 = \lambda_\Trans + \lambda_\AB + 2\da_{m,j}\da_{n,j}\db_m\db_n.$$

\noindent We then compute the Laplacian of the vector $\vec d$:

    \begin{align*}
        (\dive \grad \vec d)_i &= (\dive \grad \left(\vec d_\Trans \times \vec d_\AB\right))_i \\
            &= (\grad \left(\vec d_\Trans \times \vec d_\AB\right))_{ij,j} \\
            &= (\vec d_\Trans \times \vec d_\AB)_{i,jj} \\
            &= ([\vec d_\Trans]_k[\vec d_\AB]_\ell \epsilon_{k\ell i})_{,jj} \\
            &= ([\vec d_\Trans]_{k,jj}[\vec d_\AB]_\ell + 2[\vec d_\Trans]_{k,j}[\vec d_\AB]_{\ell,j} + [\vec d_\Trans]_k[\vec d_\AB]_{\ell,jj})\epsilon_{k\ell i}\\
            &= (2(\lambda_\Trans+\lambda_\AB)[\vec d_\Trans]_k[\vec d_\AB]_\ell + 2[\vec d_\Trans]_{k,j}[\vec d_\AB]_{\ell,j})\epsilon_{k\ell i}\\
            &= 2(\lambda_\Trans+\lambda_\AB)[\vec d]_i + 2[\vec d_\Trans]_{k,j}[\vec d_\AB]_{\ell,j}\epsilon_{k\ell i}\\
            &= 2\left(-\frac 1 2 |\grad \vec d|^2 \right)[\vec d]_i + 2[\vec d_\Trans]_{k,j}[\vec d_\AB]_{\ell,j}\epsilon_{k\ell i} - 2\da_{m,j}\da_{n,j}\db_m\db_n d_i,
    \end{align*}
from which we have recovered the $|\grad \vec d|^2$ term to build the multiplier. We have obtained a loading force $\vec F = 2[\vec d_\Trans]_{k,j}[\vec d_\AB]_{\ell,j}\epsilon_{k\ell i} + 2\da_{m,j}\da_{n,j}\db_m\db_n$, where we note that $\vec F\cdot \vec d=0$: 
\begin{align*}
    \vec F \cdot \vec d &= f_id_i \\
        &= -2\da_{m,j}\da_{n,j}\db_m\db_n + 2\da_{k,j}\db_{l,j}\epsilon_{kli}\da_p\db_q\epsilon_{pqi} \\
        &= -2\da_{m,j}\da_{n,j}\db_m\db_n + 2\da_{k,j}\db_{l,j}\da_p\db_q(\delta_{kp}\delta_{lq}-\delta_{kq}\delta_{lp}) \\
        &= -2\da_{m,j}\da_{n,j}\db_m\db_n - 2\da_{k,j}\db_{l,j}\da_l\db_k \\
        &= -2\da_{m,j}\da_{n,j}\db_m\db_n + 2\da_{k,j}\db_{l}\da_{l,j}\db_k \\
        &= 0.
\end{align*}

The proof is concluded by setting $\lambda \coloneqq -\frac 1 2 |\grad \vec d|^2$, by noting that $|\vec d|=1$, that $\lambda = -\frac 1 2|\grad \vec d|^2 + \vec F \cdot \vec d$, and that the previous computations can be written as 
        $$ -\Delta \vec d + 2\lambda\vec d = \vec F. $$
\end{proof}
The final step of this section is showing that $\vec f\coloneqq \mat Q\vec d$, where $\vec d$ is a loaded nematic liquid crystal such that its load $\vec F$ satisfies $\vec F\cdot \vec d=0$. Then, $\vec f$ is also a loaded nematic liquid crystal. We establish this result in the following Lemma, for which we require a technical hypothesis. 
\begin{lemma}\label{lemma:f}
Consider a nematic liquid crystal $\vec d$ with Lagrange multiplier $\lambda$ and an external force of $\vec F$ such that $\vec F\cdot \vec d=0$. Consider also a rotation tensor $\mat Q$ such that 
    \begin{equation}\label{eq:error-nematic}
        \left([\Delta \mat Q]^T\mat Q - \mat Q^T[\Delta \mat Q]\right): \vec d\otimes \vec d=0. 
    \end{equation}
Then, the vector field $\vec f\coloneqq \mat Q\vec d$ behaves as a loaded nematic liquid crystal with Lagrange multiplier $\lambda_f\coloneqq -\frac 1 2|\grad \vec f|^2$ and an external force given by 
    $$ \vec S_i = -Q_{ik,jj}d_k +Q_{ik}F_k - 2 Q_{ik,j}d_{k,j}-Q_{im,j}Q_{in,j}Q_{ip}d_md_nd_p-2Q_{im,j}d_md_nd_{n,j}.$$  
In particular, it holds that $\vec S\cdot \vec f = -\frac 1 2\left([\Delta \mat Q]^T\mat Q - \mat Q^T[\Delta \mat Q]\right): \vec d\otimes \vec d=0$.
\end{lemma}
\begin{proof}
 We will require the following computation:
    \begin{align*}\label{eq:Q-ortho}
       0 = \Delta (\mat Q^T\mat Q)_{ij} = (Q_{ki} Q_{kj})_{,ll} = Q_{ki,ll}Q_{kj} + 2Q_{ki,l} Q_{kj,l} + Q_{ki}Q_{kj,ll}. 
    \end{align*}
Now we compute the gradient norm:
\begin{align*}
    |\grad \vec f|^2 &= f_{i,j} f_{i,j} \\
      &= (Q_{im}d_m)_{,j}(Q_{in}d_n)_{,j} \\
      &= (Q_{im,j}d_m + Q_{im}d_{m,j})(Q_{in,j}d_n + Q_{in}d_{n,j}) \\
      &= |\grad \vec d|^2 + Q_{im,j}Q_{in,j}d_md_n + Q_{im,j}Q_{in}d_m d_{n,j} + Q_{im}Q_{in,j}d_{m,j}d_n \\
      &= |\grad \vec d|^2 + Q_{im,j}Q_{in,j}d_md_n + 2Q_{im,j}Q_{in}d_m d_{n,j} \\
      &= |\grad \vec d|^2 -\frac 1 2\left(Q_{im,jj}Q_{in} + Q_{im}Q_{in,jj}\right)d_md_n + 2Q_{im,j}Q_{in}d_m d_{n,j},
\end{align*}
and as before, we compute the Laplacian of $\vec f$:
    \begin{align*}
        (\Delta \vec f)_i = (\mat Q\vec d)_{i,jj} &= Q_{ik,jj}d_k + Q_{ik}d_{k,jj} + 2Q_{ik,j} d_{k,j} \\
                &= Q_{ik,jj}d_k + (\mat Q[\Delta \vec d])_i + 2 Q_{ik,j} d_{k,j} \\
                &= Q_{ik,jj}d_k + 2\lambda f_i - (\mat Q \vec F)_i + 2Q_{ik,j}d_{k,j} \\
                &= Q_{ik,jj}d_k + 2\left(\lambda - Q_{im,j}Q_{in,j}d_m d_n - 2Q_{im,j}Q_{in}d_m d_{n,j} \right) f_i \\
                &\qquad\qquad - (\mat Q \vec F)_i + 2Q_{ik,j}d_{k,j} + (Q_{im,j}Q_{in,j}d_m d_n + 2 Q_{im,j}Q_{in}d_m d_{n,j})Q_{ip}d_p,
    \end{align*}
where we have built the multiplier $\lambda_f=-\frac 1 2 |\grad f|^2$, and obtained the external force
    $$ -\vec S=Q_{ik,jj}d_k -(\mat Q \vec F)_i + 2Q_{ik,j}d_{k,j} + (Q_{im,j}Q_{in,j}d_m d_n + 2Q_{im,j}Q_{in}d_m d_{n,j})Q_{ip}d_p. $$
Its projection onto $\vec f$ is given by 
    \begin{align*}
        -\vec S\cdot \vec f&=Q_{ik,jj}d_k Q_{ip}d_p-(\mat Q \vec F)_i(\mat Q\vec d)_i + 2Q_{ik,j}d_{k,j}Q_{ip}d_p + Q_{im,j}Q_{in,j}d_m d_n + 2Q_{im,j}Q_{in}d_m d_{n,j} \\
            &= Q_{ik,jj}d_k Q_{ip}d_p + 2Q_{ik,j}d_{k,j}Q_{ip}d_p + Q_{im,j}Q_{in,j}d_m d_n - 2Q_{im}Q_{in,j}d_m d_{n,j} \\
            &= Q_{ik,jj}d_k Q_{ip}d_p + Q_{ik,j}Q_{ip,j}d_k d_p \\ 
            &= \frac 1 2 \left(Q_{ik,jj}Q_{ip}- Q_{ik}Q_{ip,jj}\right)d_k d_p.
    \end{align*}
As in the previous case, we see that the fiber field $\vec f$ is exactly a loaded nematic liquid crystal if $\vec S\cdot \vec f=0$, which holds if and only if
    $$ \left([\Delta \mat Q]^T\mat Q - \mat Q^T[\Delta \mat Q]\right): \vec d \otimes \vec d=0,$$
i.e.  hypothesis \eqref{eq:error-nematic}. This concludes the proof.
\end{proof}

The main difficulty in computing the vector field $\vec f$ is that we require a transmurally varying weight to obtain an explicit rotation matrix that interpolates the known boundary values. Luckily, the Frank-Oseen equations naturally yield  a vector interpolation method, which we show analytically in Section \ref{section:rotation}. In practice, this implies that an efficient solution for avoiding the computation of $\mat Q$ in all the domain is to use the expression $\mat Q\vec d$ as a boundary condition (where the values of $\mat Q$ are known), and consider a Frank-Oseen problem assuming that $[\Delta \mat Q]^T\mat Q - \mat Q^T[\Delta \mat Q]\perp \vec d\otimes \vec d$.
This justifies the practical importance of Lemma \ref{lemma:f}. We verify this assumption in Section \ref{section:application}, and show analytically the unit-vector interpolation property in Section \ref{section:rotation}. We highlight that hypothesis \eqref{eq:error-nematic} is not to be intepreted as a weakness of the Frank-Oseen equations, but as the required hypothesis under which RBMs yield a nematic liquid crystal, as expected from physical observations.

\subsection{The Frank-Oseen model as a unit vector interpolation method}\label{section:rotation}

The scope of this section is to show that the Frank-Oseen equations yield the same vector rotations used in quaternion interpolation, which substitutes the role of the weight function in RBMs. We provide three examples: (i) a 1D model in which we recover the standard quaternion interpolation formula--\emph{slerp}--analytically, (ii) a 2D case in which we show the mechanism through with the Frank-Oseen theory generates the apex singularity and (iii) a numerical example with more complex spatial interactions, where we show that interpolation velocity can behave nonlinearly according to the boundary conditions.

\subsubsection{Slerp in a 1D domain}
We restrict our analysis to the one-dimensional case. For this, we consider that the fiber field is of the form $\vec f=(u,v,0)=:(\vec d, 0)$, and that it is constant along the $y$ and $z$ axes. This transforms \eqref{eq:fiber-saddle} into the following system of ODEs in an interval $I=(0,1)$:
\begin{equation}\label{eq:fo-1d}
    \begin{aligned}
    -u'' + 2\lambda u &= 0 \quad \text{ in $I$},\\
    -v'' + 2\lambda v &= 0 \quad \text{ in $I$},\\
    u^2 + v^2        &= 1 \quad \text{ in $I$},\\
    \vec d(0) &= \vec a, \\
    \vec d(1) &= \vec b, \\
    \end{aligned}
\end{equation}
for given unit vectors $\vec a, \vec b$ such that $\vec a\neq \vec b$. Using standard Ordinary Differential Equations theory, we propose a solution of the form
$$u(x) = C_1\cosb{\omega x} + C_2\sinb{\omega x},$$
$$ v(x) = D_1\cosb{\omega x} + D_2\sinb{\omega x}, $$
where $\lambda = -\frac 1 2\omega^2$. After some algebraic manipulations, the first boundary condition yields $(C_1, D_1) = \vec a$, and the second one yields $(C_2, D_2) = [\sinb \omega]^{-1}(\vec b - \cosb \omega \vec a)$. This means that the solution can be written as
\begin{align*}
\vec d &= [\sinb\omega]^{-1}\left([\sinb{\omega} \cosb{\omega x} - \cosb{\omega} \sinb{\omega x} ]\vec a+\sinb{\omega x}\vec b\right),
\end{align*}
which using $\sinb{\alpha-\beta}=\sinb \alpha \cosb \beta - \cosb \alpha \sinb \beta$ can be reduced to
\begin{equation}\label{eq:fo-is-rotation}
\vec d(x) = \frac{\sinb{\omega[1-x]}}{\sinb{\omega}}\vec a + \frac{\sinb{\omega x}}{\sinb{\omega}}\vec b.
\end{equation}
This is exactly the formula used for quaternion interpolation if $\omega$ were the angle subtended by $\vec a$ and $\vec b$, as we will see. This formula is also valid for the interpolation of elements in $\mathbb S^2$, as well as unit quaternions \cite{AnimatingRotatShoema1985}. We finally compute the value of $\omega$ by imposing the unit norm constraint:
$$\sinsq{\omega} = \sinsq{w[1-x]} + \sinsq{\omega x} + 2\sinb{w[1-x]}\sinb{\omega x}\langle \vec a,\vec b\rangle.$$
We note that we can write 
$$ \sinb{\omega} = \sinb{\omega - \omega x + \omega x} = \sinb{\omega[1-x]}\cosb{\omega x} + \cosb{\omega[1-x]}\sinb{\omega x},$$
which together with the unit norm constraint gives the following:
$$\sinb{\omega x}\sinb{\omega[1-x]}\left(\cosb{\omega} - \langle \vec a,\vec b\rangle\right)=0.$$
As this equality holds for every $x$, and $\omega\neq 0$ (unless $\vec a=\vec b$), we obtain that 
\begin{equation}\label{eq:omega}
\cos \omega = \langle \vec a, \vec b\rangle,
\end{equation}
as expected. This fundamental result concludes our claim, as $\vec a$ and $\vec b$ are unit vectors. We illustrate this result in Figure \ref{fig:fo-rotation}, where we compare the interpolated vector field for two given boundary vectors with (a) a Laplace equation, (b) the Laplace equation normalized, and (c) the Frank-Oseen solution from \eqref{eq:fo-1d}. Note in particular that arbitrarily normalizing can give rise to spurious singularities, and in fact to compute the vector field in (b), the normalization was modified with a small constant $\epsilon=10^{-8}$ using $\Pi_{\mathbb S^2}^\epsilon$.

\begin{lemma}
The solution of \eqref{eq:fo-1d} is given by \eqref{eq:fo-is-rotation}, where $\omega$ is given by \eqref{eq:omega}.
\end{lemma}

\begin{figure}
    \centering
    \usetikzlibrary {arrows.meta}
    \def\scale{0.6}
    \def\samples{12}
    \def\lw{1.1pt} 
    \centering
    \begin{subfigure}{0.33\textwidth}
        \def\V{\scale*t)}
         \begin{tikzpicture}
         \begin{axis}[width=\textwidth, xmin=-0.1, xmax=1.1, ymin=-1, ymax=1, axis x line=none, axis y line=none, xtick=\empty]
             \addplot[blue!80!black, point meta={abs(\V)}, line width=\lw, domain=-1:1, variable=\t, quiver={u=0.0,v=\V, every arrow/.append style={ line width=\lw*\pgfplotspointmetatransformed/1000 }}, 
                    -{Latex}, samples=\samples] ({(t/2 + 0.5}, {0.0});
             \draw[black, -] (axis cs:-0.1,0) -- (axis cs:1.1,0);
        \end{axis}
        \end{tikzpicture}
    \caption{Laplace.}
    \end{subfigure}
    \begin{subfigure}{0.33\textwidth}
         \begin{tikzpicture}
         \begin{axis}[width=\textwidth, xmin=-0.1, xmax=1.1, ymin=-1, ymax=1, axis x line=none, axis y line=none, xtick=\empty]
             \addplot[blue!80!black, line width=\lw, domain=-1:1, variable=\t, quiver={u=0.0,v=\scale*t/abs(1e-12 + t}, -{Latex}, samples=\samples] ({(t/2 + 0.5}, {0.0});
             \draw[black, -] (axis cs:-0.1,0) -- (axis cs:1.1,0);
        \end{axis}
        \end{tikzpicture}
    \caption{Normalized Laplace.}
    \end{subfigure}
    \begin{subfigure}{0.33\textwidth}
         \begin{tikzpicture}
         \begin{axis}[width=\textwidth, xmin=-0.6*pi, xmax=0.6*pi, ymin=-1, ymax=1, axis x line=none, axis y line=none, xtick=\empty]
             \addplot[blue!80!black, line width=\lw, domain=-0.5*pi:0.5*pi, variable=\t, quiver={u={1.5*\scale*cos(deg(t))},v={\scale*sin(deg(t))}}, -{Latex}, samples=\samples] ({t}, {0.0});
             \draw[black, -] (axis cs:-0.6*pi,0) -- (axis cs:0.6*pi,0);
        \end{axis}
        \end{tikzpicture}
    \caption{Frank-Oseen.}
    \end{subfigure}
    \caption{Comparison of vector interpolation schemes. In (a) we interpolate the boundary values with a Laplace equation. The solution of (a) is then normalized in (b). In (c) we show the solution of the Frank-Oseen equations.}
    \label{fig:fo-rotation}
\end{figure}
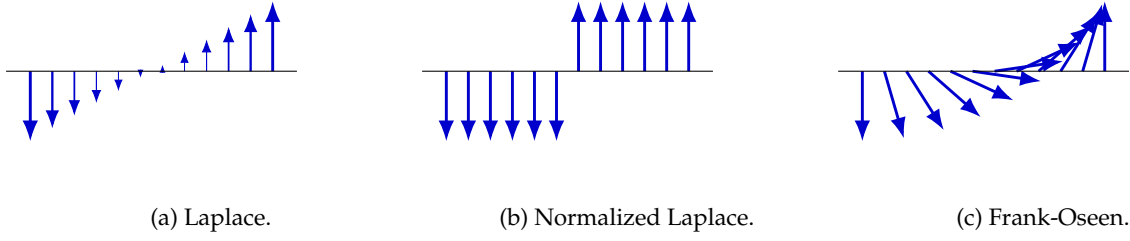

\begin{remark}
We note that the previous analysis and conclusions are also valid in a 1D segment in 3D as long as the vector field is constant along the directions orthogonal to the segment. 
\end{remark}

\subsubsection{A harmonic example in 2D}\label{section:harmonic-2d}
We now consider a circular domain $\Omega=\{\vec x: |\vec x|\leq 1\}$ and a singular solution of the Frank-Oseen equations given by 
    $$\vec d = \frac{\vec x}{|\vec x|}.$$
Consider a constant rotation matrix $\mat Q$ with $\mat Q^T\mat Q=\mat I$, then the vector $\mat Q\vec d$ is still a solution of the Frank-Oseen equation:
    $$ \vec 0 = \mat Q(-\Delta \vec d + \lambda \vec d) = -\Delta \left(\mat Q\vec d\right) + \lambda \left(\mat Q\vec d\right),$$
where the unitary norm constraint is naturally satisfied, and the boundary conditions need to be rotated as well. We note that these computations are locally equivalent to the ones performed on the epicardium when computing transversal vector field and then when rotating it to compute the fiber field. We depict the singular solution, a rotated solution, and the transversal solution in Figure \ref{fig:analytic-2d}, all generated as rotations of the singular radial function $\vec d$. In addition, we show in the second row of Figure \ref{fig:analytic-2d} the restriction of the solutions to a horizontal arc passing through the origin, and note that the vector interpolation obtained through this line is piecewise constant and discontinuous, in constrast to the analytic 1D example.

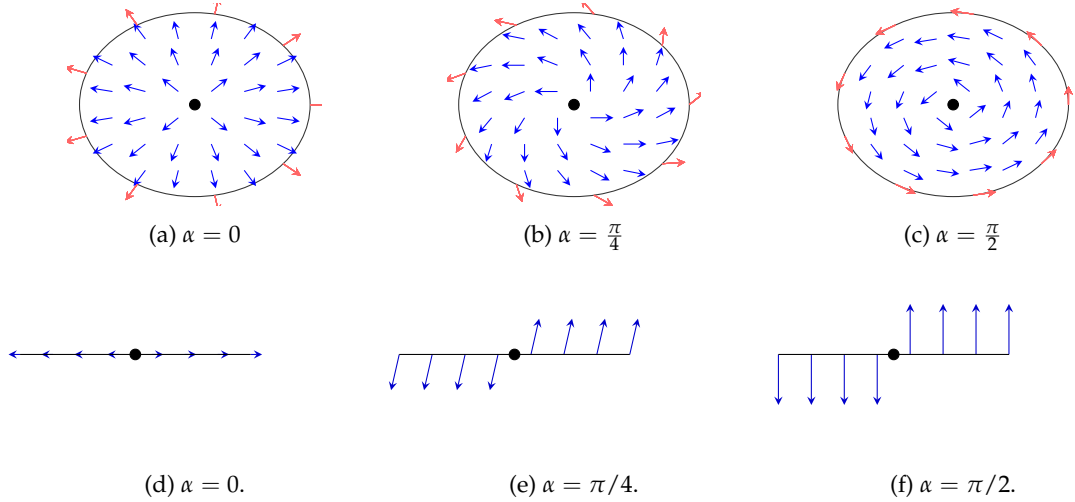
\begin{figure}
    \centering
    \def\figwid{0.3}
    \def\lw{1.1pt}
    \def\scale{0.6}
    \def\nrm{sqrt(x^2+y^2)}
    \def\denom{(1e-6 + sqrt(x^2+y^2))}
    \def\samplesblue{8}
    \def\samplesred{10}
    \def\scalearrow{0.2}
    \begin{subfigure}{\figwid\textwidth}
        \centering
        \begin{tikzpicture}
        \begin{axis}[
            domain=-1:1, 
            width=\textwidth,
            xmin=-1.1,   xmax=1.1,
            ymin=-1.1,   ymax=1.1,
            axis x line=none, axis y line=none, xtick=\empty,
            view={0}{90},
            axis background/.style={fill=white},
        ]
            \draw[black!80!white] \pgfextra{
              \pgfpathellipse{\pgfplotspointaxisxy{0}{0}}
                {\pgfplotspointaxisdirectionxy{1}{0}}
                {\pgfplotspointaxisdirectionxy{0}{1}}
            };
            \addplot3 [only marks,mark=*] coordinates { (0,0,0) };
            \addplot3[blue,
			quiver={
                u={x/\denom},
                v={y/\denom},
			    scale arrows=\scalearrow,
			},
            -stealth,samples=\samplesblue]
				{\nrm <= 0.9 ? 1 : NaN};
             \addplot3[red!60!white, domain=0:2*pi, variable=\t, quiver={u={x/\denom},v={y/\denom},w={0}, scale arrows=\scalearrow}, -stealth, samples=\samplesred] ({cos(deg(t))}, {sin(deg(t))}, {0});
        \end{axis}
        \end{tikzpicture}
        \caption{$\alpha=0$}
    \end{subfigure}
    \begin{subfigure}{\figwid\textwidth}
        \centering
        \begin{tikzpicture}
        \begin{axis}[
            domain=-1:1, 
            width=\textwidth,
            xmin=-1.1,   xmax=1.1,
            ymin=-1.1,   ymax=1.1,
            axis x line=none, axis y line=none, xtick=\empty,
            view={0}{90},
            axis background/.style={fill=white},
        ]
            \draw[black!80!white] \pgfextra{
              \pgfpathellipse{\pgfplotspointaxisxy{0}{0}}
                {\pgfplotspointaxisdirectionxy{1}{0}}
                {\pgfplotspointaxisdirectionxy{0}{1}}
            };
            \addplot3 [only marks,mark=*] coordinates { (0,0,0) };
            \addplot3[blue,
			quiver={
                u={cos(deg(pi/4))*x/\denom-sin(deg(pi/4))*y/\denom},
                v={sin(deg(pi/4))*x/\denom+cos(deg(pi/4))*y/\denom},
			    scale arrows=\scalearrow,
			},
            -stealth,samples=\samplesblue]
				{\nrm <= 1 ? 1 : NaN};
             \addplot3[red!60!white, domain=0:2*pi, variable=\t, quiver={u={cos(deg(pi/4))*x/\denom-sin(deg(pi/4))*y/\denom},v={sin(deg(pi/4))*x/\denom+cos(deg(pi/4))*y/\denom},w={0}, scale arrows=\scalearrow}, -stealth, samples=\samplesred] ({cos(deg(t))}, {sin(deg(t))}, {0});
        \end{axis}
        \end{tikzpicture}
        \caption{$\alpha=\frac \pi 4$}
    \end{subfigure}
    \begin{subfigure}{\figwid\textwidth}
        \centering
        \begin{tikzpicture}
        \begin{axis}[
            domain=-1:1, 
            width=\textwidth,
            xmin=-1.1,   xmax=1.1,
            ymin=-1.1,   ymax=1.1,
            axis x line=none, axis y line=none, xtick=\empty,
            view={0}{90},
            axis background/.style={fill=white},
        ]
            \draw[black!80!white] \pgfextra{
              \pgfpathellipse{\pgfplotspointaxisxy{0}{0}}
                {\pgfplotspointaxisdirectionxy{1}{0}}
                {\pgfplotspointaxisdirectionxy{0}{1}}
            };
            \addplot3 [only marks,mark=*] coordinates { (0,0,0) };
            \addplot3[blue,
			quiver={
                u={cos(deg(pi/2))*x/\denom-sin(deg(pi/2))*y/\denom},
                v={sin(deg(pi/2))*x/\denom+cos(deg(pi/2))*y/\denom},
			    scale arrows=\scalearrow,
			},
            -stealth,samples=\samplesblue]
				{\nrm <= 1 ? 1 : NaN};
             \addplot3[red!60!white, domain=0:2*pi, variable=\t, quiver={u={cos(deg(pi/2))*x/\denom-sin(deg(pi/2))*y/\denom},v={sin(deg(pi/2))*x/\denom+cos(deg(pi/2))*y/\denom},w={0}, scale arrows=\scalearrow}, -stealth, samples=\samplesred] ({cos(deg(t))}, {sin(deg(t))}, {0});
        \end{axis}
        \end{tikzpicture}
        \caption{$\alpha=\frac \pi 2$}
    \end{subfigure}

    \begin{subfigure}{\figwid\textwidth}
         \begin{tikzpicture}
         \begin{axis}[width=\textwidth, xmin=-1.1, xmax=1.1, ymin=-0.2, ymax=0.2, axis x line=none, axis y line=none, xtick=\empty]
             \addplot[blue!80!black, domain=-1:1, variable=\t, quiver={u={t/\denom},v={0},scale arrows=0.1}, -stealth, samples=\samplesblue] ({t}, {0});
             \addplot [only marks,mark=*] coordinates { (0,0) };
             \draw[black, -] (axis cs:-1,0) -- (axis cs:1,0);
        \end{axis}
        \end{tikzpicture}
    \caption{$\alpha=0$.}
    \end{subfigure}
    \begin{subfigure}{\figwid\textwidth}
         \begin{tikzpicture}
         \begin{axis}[width=\textwidth, xmin=-1.1, xmax=1.1, ymin=-0.2, ymax=0.2, axis x line=none, axis y line=none, xtick=\empty]
             \addplot[blue!80!black, domain=-1:1, variable=\t, quiver={u={cos(deg(pi/4))*t/\denom},v={sin(deg(pi/4))*t/\denom},scale arrows=0.1}, -stealth, samples=\samplesblue] ({t}, {0});
             \addplot [only marks,mark=*] coordinates { (0,0) };
             \draw[black, -] (axis cs:-1,0) -- (axis cs:1,0);
        \end{axis}
        \end{tikzpicture}
    \caption{$\alpha=\pi/4$.}
    \end{subfigure}
    \begin{subfigure}{\figwid\textwidth}
         \begin{tikzpicture}
         \begin{axis}[width=\textwidth, xmin=-1.1, xmax=1.1, ymin=-0.2, ymax=0.2, axis x line=none, axis y line=none, xtick=\empty]
             \addplot[blue!80!black, domain=-1:1, variable=\t, quiver={u={cos(deg(pi/2))*t/\denom},v={sin(deg(pi/2))*t/\denom},scale arrows=0.1}, -stealth, samples=\samplesblue] ({t}, {0});
             \addplot [only marks,mark=*] coordinates { (0,0) };
             \draw[black, -] (axis cs:-1,0) -- (axis cs:1,0);
        \end{axis}
        \end{tikzpicture}
    \caption{$\alpha=\pi/2$.}
    \end{subfigure}
    \caption{Harmonic example in 2D with singularity. In the first row, we show the well-known hedgehog singular solution (a) and then rotate it in $\pi/4$ and $\pi/2$, as shown in (b) and (c) respectively. The second row displays a line that passes through the middle of each corresponding figure in the first row.}
    \label{fig:analytic-2d}
\end{figure}

\subsubsection{Local minima in vector interpolation}
In this section we provide some examples of local minima, where another solutions can be found analytically for the rotation example shown in Section \ref{section:harmonic-2d}, but with higher energy. Using Lemma \ref{lemma:f}, it is not difficult to see that the function $\vec u(r, \theta)\coloneqq \mat Q_{\alpha(r)}\vec d_0(\theta)$ is a nematic liquid crystal, with its components defined as 

$$ Q_{\alpha(r)}\coloneqq \begin{bmatrix} \cosb{\alpha(r)} & -\sinb{\alpha(r)} \\ \sinb{\alpha(r)} & \cosb{\alpha(r)} \end{bmatrix}, \vec d_0(\theta) \coloneqq \begin{bmatrix} -\sinb{\theta} \\ \cos{\theta} \end{bmatrix}.$$
We consider this problem in a ring of internal and external radii of $0<\rho<R$, with the boundary conditions given by an external rotation of $0$ and an intenral rotation of $\alpha_0$. Indeed, after some elementary calculations we obtain that $\alpha(r) = \alpha_0 \frac{\log\left(r/R\right)}{\log\left(\rho/R\right)}.$
We show the solution of this problem for a rotation of $\pi$ and $3\pi$ in Figure \ref{fig:non-uniqueness}. 

\begin{figure}
    \centering
    \begin{subfigure}{0.49\textwidth}
        \centering
        \includegraphics[width=0.6\textwidth]{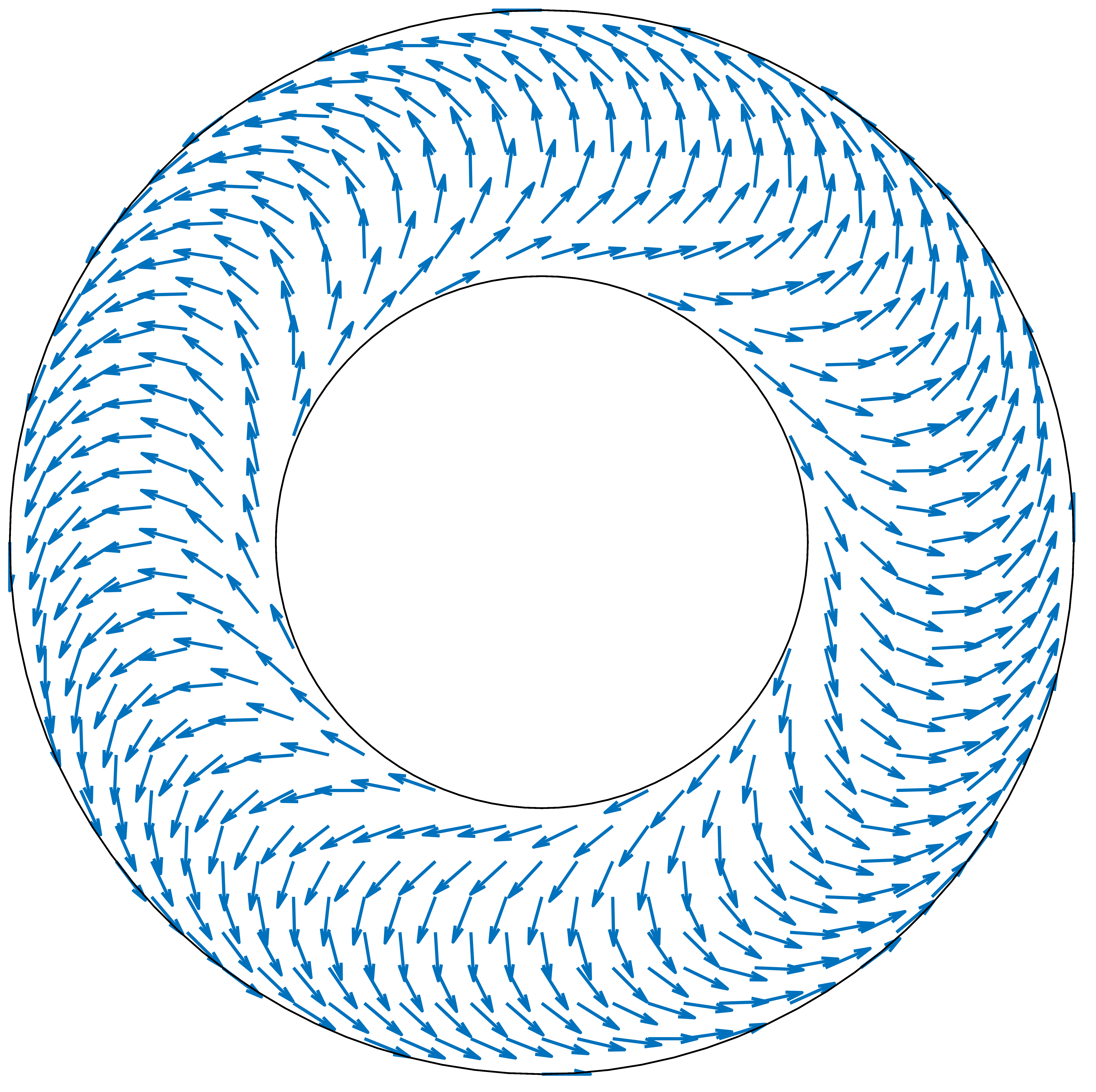}
        \caption{Rotation angle of $\pi$.}
    \end{subfigure}
    \begin{subfigure}{0.49\textwidth}
        \centering
        \includegraphics[width=0.6\textwidth]{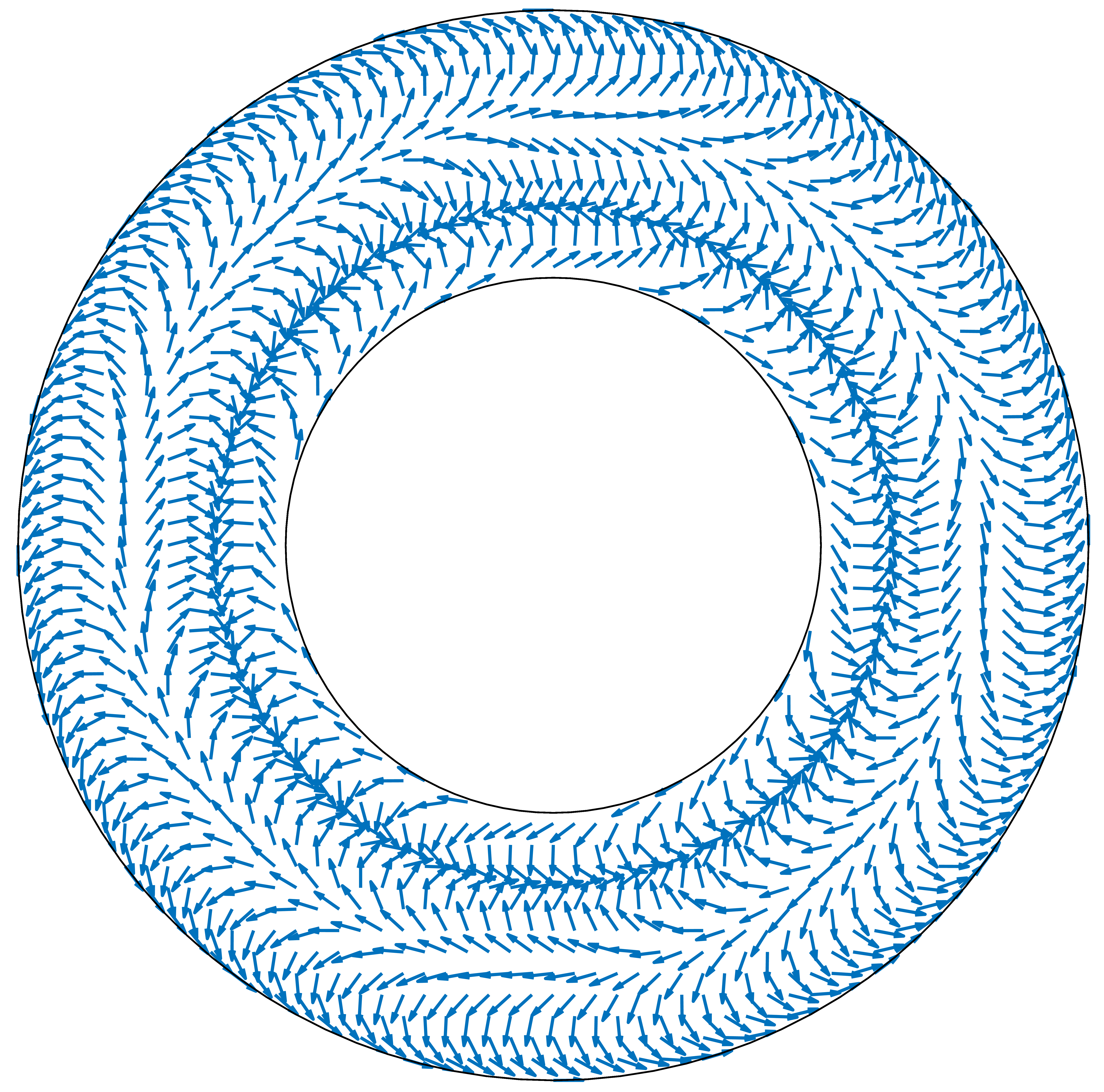}
        \caption{Rotation angle of $3\pi$.}
    \end{subfigure}
    \caption{Display of local minima, where in (a) we show the solution that behaves as in the 1D example for $\alpha_0=\pi$. In (b), we show that the solution can present additional rotations, where the point where the rotation becomes $2\pi$ has been depicted by a blue ring within the geometry (a).}
    \label{fig:non-uniqueness}
\end{figure}

\section{Solution methodology and numerical approximation}\label{section:methodology}
The scope of this section is twofold. On one hand, we show how to use Nitsche's trick to impose the boundary condition on the normal component. On the other one, we propose a preconditioned projected gradient descent method for solving the Frank-Oseen equations. As a result, we provide a complete overview of how this model can be implemented in practice. 
Throughout this section we will denote with $\mathbb P_k$ the space of $k$-th order finite elements. All of our numerical tests have been performed using the Firedrake library \cite{rathgeber2016firedrake} with the Algebraic Multigrid implementation from HYPRE, BoomerAMG \cite{falgout2002hypre}. 

\subsection{Imposing the boundary condition $\vec d\cdot \vec N=0$}
Complex geometries do not allow for an explicit computation of the normal vector $\vec N$, because of which imposing this condition is non trivial. We rely on Nitsche's trick to do this \cite{stenberg1995some}, which we derive for this model in the following. We note that this is very similar to the slip boundary condition used in Navier-Stokes equations \cite{FiniteElementVerfr1986}, with a Nitsche formulation available in \cite{NitschesMethoGjerde2021}. 

Define the functional spaces
    $$   \vec V_{\vec N} = \{\vec y \in \vec H^1(\Omega): \vec y = \vec N \quad \text{on $\Gamma_{D, \vec \tau}^{\vec d}$}\},\qquad \vec V_{\vec 0} = \{\vec y \in \vec H^1(\Omega): \vec y = \vec 0 \quad \text{on $\Gamma_{D, \vec \tau}^{\vec d}$}\}, \\
    $$
with which we now perform integracion by parts for the Laplace operator for $\vec d$ in $\vec V_{\vec N}$ and a test function $\vec v$ in $\vec V_{\vec 0}$: 
    $$ -\int_\Omega \Delta \vec d\cdot \vec v\,dx = \int_\Omega \grad \vec d:\grad \vec v\,dx - \int_{\Gamma_{D, \vec N}^{\vec d}} \vec v\cdot ([\grad \vec d]\vec N)\,ds = \int_\Omega \grad \vec d:\grad \vec v\,dx - \int_{\Gamma_{D, \vec N}^{\vec d}} \vec v\cdot ([\vec N\otimes \vec N][\grad \vec d]\vec N)\,ds,  $$
where the second equality comes from the Neumann boundary condition $ \Pi_{\vec \tau}[\grad \vec d]\vec N = \vec 0$. Restating  $\vec d\cdot \vec N=0$ as $[\vec N\otimes \vec N]\vec d = \vec 0$, to preserve the symmetry and ellipticity of the formulation we consider the following additional terms:
    $$ \int_{\Gamma_{D, \vec N}^{\vec d}} \left([\vec N\otimes \vec N]\vec d\right) \cdot \left([\grad \vec v]\vec N\right)\,ds, C\sum_{T\in \mathcal T_h}\int_{T}h_T^{-1}(\vec d\cdot N)(\vec v\cdot \vec N)\,ds,$$
where $\mathcal T_h$ stands for the geometry triangulation, $T$ is each element of the mesh, and $h_T$ is the element diameter.  These terms guarantees that the problem is elliptic for $C>0$ sufficiently large \cite{stenberg1995some}, where we have observed a value of $C=10$ to yield satisfactory results in practice. The resulting weak form of the Laplace operator is given by
    \begin{multline}\label{eq:laplace-nitsche}
            -\int_\Omega \Delta \vec d\cdot \vec v\,dx = \int_\Omega \grad \vec d:\grad \vec v\,dx - \int_{\Gamma_{D, \vec N}^{\vec d}} \left([\vec N\otimes \vec N]\vec v\right)\cdot ([\grad \vec d]\vec N)\,ds\\
            + \int_{\Gamma_{D, \vec N}^{\vec d}} \left([\vec N\otimes \vec N]\vec d\right) \cdot \left([\grad \vec v]\vec N\right)\,ds 
            + C\sum_{T\in \mathcal T_h}\int_{T}h_T^{-1}(\vec d\cdot N)(\vec v\cdot \vec N)\,ds\qquad\forall \vec v\in \vec V_{\vec 0}.
    \end{multline}

\subsection{The solution strategy}
We propose a solution strategy for solving problem \eqref{eq:fiber-min}. As we will show, standard Newton-based solvers for the first order conditions \eqref{eq:fiber-saddle} are not robust with respect to the initial condition. To alleviate this difficulty, we use a preconditioned projected gradient descent \cite{wright1999numerical} , and show that it is effective for solving both singularities and arbitrary initial conditions. Another disadvantage of the saddle-point formulation is that it requires the use of inf-sup stable discrete spaces \cite{boffi2013mixed}. This means that the vector must always be approximated with at least second order finite elements, as shown in \cite{ConstrainedOptAdler2016}.  The convergence of a finite element approximation has been shown in \cite{ConstrainedOptAdler2016} for the saddle-point formulation, but a similar approximation result for the minimization problem \eqref{eq:fiber-min} has not yet been established. 

\begin{remark}
The criticism of the saddle point formulation is slightly unfair. The use of higher order discretizations can be avoided using adequate stabilization terms or by devising an Uzawa scheme. Still, both cases require computing the Lagrange multiplier as an additional variable, which makes the model more expensive that necessary.
\end{remark}

We consider an initial guess given by a vector $\vec d^0$ such that $|\vec d^0|=1$ throughout the domain. We note that the gradient of $\Psi$ defined in \eqref{eq:fiber-min} is given by 
    $$ d\Psi(\vec d) = -\Delta \vec d, $$
which does not take into account the effect of the unit norm constraint. Recalling that $\lambda=-\frac 1 2|\grad \vec d|^2$, the residual of \eqref{eq:fiber-min} is given by
    $$ d\widetilde{\Psi}(\vec d) = -\Delta \vec d - |\grad \vec d|^2 \vec d, $$
which is a difficult problem as, in addition to it being nonlinear, the negative sign of $-|\grad \vec d|^2$ breaks the ellipticity. Because of this, we will consider a quasi-Newton inspired projected gradient descent, where we use $d\Psi$ instead of $d\widetilde{\Psi}$ for computing the Jacobian approximation, as it is a well understood operator. We show the resulting algorithm in Algorithm \ref{alg:ppgd}. The following comments are in place:
\begin{algorithm}[!ht]
\caption{Preconditioned projected gradient descent algorithm.}
\label{alg:ppgd}
\begin{algorithmic}[1]
  \STATE {\bf Input:} Initial point $\vec d^0$, tolerance \texttt{tol} and maximum iterations \texttt{maxit}
  \STATE Set $\texttt{error}=1$, $\texttt{k} = 0$, $\vec d^k = \vec d^0$
  \WHILE{$\texttt{error}>\texttt{tol}$ and $\texttt{k}<\texttt{maxit}$}
  \STATE Compute $\delta \vec d^k$ such that $\ten P\delta \vec d^k = -d\widetilde{\Psi}(\vec d^{k-1})$
  \STATE Update the current solution $\hat{\vec d^k} = \vec d^{k-1} + \delta \vec d^k$
  \STATE Project the current solution back to the solution space $\vec d^k = \Pi_{\mathbb S^2}^\epsilon(\hat{\vec d^k})$
  \STATE $\texttt{error} = \|d\widetilde{\Psi}(\vec d^k)\|_{\ell^2}$, $\texttt{k} = \texttt{k}+1$
  \ENDWHILE
\RETURN{Solution $\vec d = \vec d^k$}
\end{algorithmic}
\end{algorithm}

\begin{itemize}
    \item Given that we consider a possibly singular solution, the accuracy of the method can be easily improved using mesh adaptivity in the singularities, where $|\grad \vec d|$ goes to infinity.
    \item We have observed that a convergence criterion given by a relative gradient norm of $\|d\widetilde{\Psi}(\vec d^{k+1})\|_{\ell^2}\leq 10^{-8}\|d\widetilde{\Psi}(\vec d^{0})\|_{\ell^2}$ yields satisfactory results, where $\|\vec a\|_{\ell^2}\coloneqq \frac 1 N\sqrt{\sum_i a_i^2}$.
    \item The preconditioner $\ten P$ should be a good approximation of the Laplace operator $d\Psi$. For this, we consider the action of a multigrid preconditioner, which is known to be optimal \cite{MultigridMethoSchobe1999}.
    \item The Jacobian is the Laplace operator and not the Hessian of the constrained problem. Albeit unintuitive, we have observed this approach to be better than the standard one as it avoids unnecessary matrix reassembly.
    \item We have chosen to use only the action of a preconditioner instead of a full-blown linear solver for the increment equation. In our experience, nonlinear problems usually do not require the solution of tangent problems with high-precision, as they can \emph{oversolve} the linearized problem without improving the original nonlinear problem \cite{ParallelInexacBarnaf2022,barnafi2022analysis}. An additional benefit of this choice is that we have removed all parameters related to the linear solver, which makes our method less sensitive to parameter-tuning.
    \item Both $\Psi$ and $\widetilde{\Psi}$ and understood as being approximated with Nitsche's trick as in \eqref{eq:laplace-nitsche}.
\end{itemize}

Given that we will use this solver for complex geometries and boundary conditions, we would like that our solver is always able to yield a solution, i.e. robust. For this reason, we compare our solver with a state-of-the-art solver \cite{AugmentedLagraXiaJ2021} in two scenarios, one where we interpolate two boundary vectors for varying initial conditions, and another one where we try to solve a problem with a singularity. The considered solver is based on the saddle point formulation \eqref{eq:fiber-saddle}, and it consists in a Newton-Krylov scheme for an augmented Lagrangian formulation with a modified Jacobian. The preconditioner is given by a block-partitioned multigrid. We use their suggested value of $10^6$ for the augmented Lagrangian parameter. To get a better grasp of the difference in complexity of both solvers, we report the PETSc options used to set each of them. 

We conclude this section by studying the performance of the solver with respect to an increasing number of degrees of freedom, to see if it is adequate for its usage in an HPC environment. All tests contained in this section are performed in a unit square $\Omega=(0,1)^2$.


\paragraph{Sensitivity to initial conditions.} We set $\vec d= (0,-1)$ on the left boundary, $\vec d = (0,1)$ on the right boundary, and homogeneous Neumann conditions on the top and bottom. We consider an initial vector $\vec d^0=(\cos \theta, \sin \theta)$ for $\theta\in [0,2\pi]$, and plot the number of nonlinear iterations in Figure \ref{fig:initial-conditions}, with a mesh discretization of 20 elements per side. The missing points show that a solver did not converge, which only happened for the monolithic Newton-Krylov solver. In fact, our method converges for all the angles considered, whose iterations do not vary significantly between first and second order finite elements.

\begin{figure}[ht!]
    \centering
    \begin{tikzpicture}
        \begin{axis}[xlabel=$\theta$, ylabel=Nonlinear iters., width=0.7\textwidth, height=6cm, legend style={fill opacity=0.6, legend cell align=left, legend pos=north west}, legend columns=1]
            \addplot[red!70!black, style={ mark options={draw opacity=1.0, fill opacity=1.0, mark size=1pt}}, mark=*, line width=1.5] coordinates
            {(0.0, 20) (1.0, 20) (1.1, 19) (1.2, 20) (1.3,21) (1.4, 22) (1.5, 22) (1.6, 20) (1.7, 22) (1.8, 21) (1.9, 20) (2.0, 18) (2.1, 20) (3.10, 20) };
            \addplot[blue!70!black, style={ mark options={draw opacity=1.0, fill opacity=1.0, mark size=1pt}}, mark=*, line width=1.5] coordinates
            {(0.0, 18) (1.1, 18) (1.2, 19) (1.9, 19) (2.0,18) (3.10, 18)};
            \addplot[green!70!black, style={ mark options={draw opacity=1.0, fill opacity=1.0, mark size=1pt}}, mark=*, line width=1.5] coordinates
            {(0.0, 11) (0.1, 12) (0.2, 12) (0.3, 14) (0.4, 14)};
            \addplot[green!70!black, style={ mark options={draw opacity=1.0, fill opacity=1.0, mark size=1pt}}, mark=*, line width=1.5] coordinates
            {(0.9, 16)};
            \addplot[green!70!black, style={ mark options={draw opacity=1.0, fill opacity=1.0, mark size=1pt}}, mark=*, line width=1.5] coordinates
            {(2.0, 23)};
            \addplot[green!70!black, style={ mark options={draw opacity=1.0, fill opacity=1.0, mark size=1pt}}, mark=*, line width=1.5] coordinates
            {(2.5,24)};
            \addplot[green!70!black, style={ mark options={draw opacity=1.0, fill opacity=1.0, mark size=1pt}}, mark=*, line width=1.5] coordinates
            {(2.8, 20) (2.9, 12) (3.1, 12)};
            \legend{PPGD $\mathbb P_1$, PPGD $\mathbb P_2$, Monolithic}
        \end{axis}
    \end{tikzpicture}
    \caption{{\bf Sensitivity to initial conditions.} Number of nonlinear iterations incurred by the preconditioned projected gradient descent (PPGD) with first and second finite elements, and the monolithic solution of the saddle point problem with discretized with $[\mathbb P_2]^k \times \mathbb P_1$ finite elements. Missing points imply that the nonlinear solver did not converge.}
    \label{fig:initial-conditions}
\end{figure}
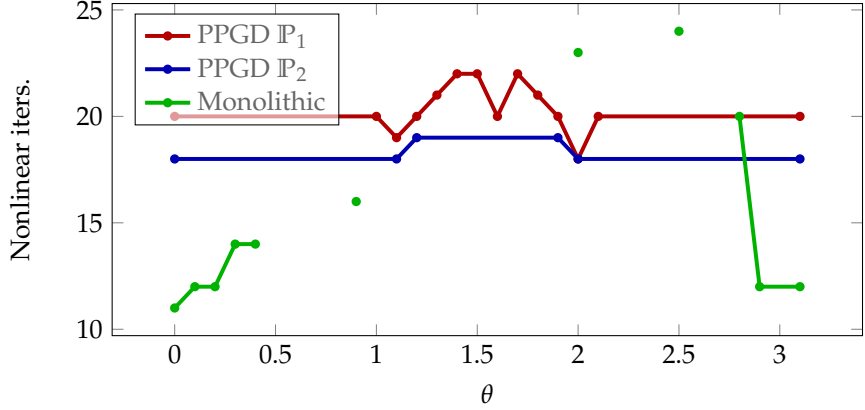

\paragraph{Singularity approximation.} This test has a similar setting to the previous one. The difference is that the top and bottom boundary conditions are also Dirichlet, with values $\vec d=(-1,0)$ and $\vec d=(1,0)$ respectively. This forces the appearance of a singularity in the center of the geometry as shown in Figure \ref{fig:singularity}. We perform the same sensitivity with respect to the initial condition in Figure \ref{fig:initial-conditions-singularity}, where we note that the saddle point model did not converge for any of the tested points. We also highlight that using 20 elements per side means that the point $(0.5, 0.5)$ is considered exactly in the model, which further increases the difficulty of this test. Indeed, the iteration numbers are much higher than in the previous case, but nevertheless the proposed solver is capable of solving the problem for all initial conditions with first order finite elements. Interestingly, this is not the case for second order elements, which presents 3 points in which the solver oscillated until the maximum number of allowed iterations was achieved. We note that the proposed strategy still has room for improvement, as we are not using any kind of stabilization technique, such as relaxation or line search. Nevertheless, it is interesting to note that this difference in performance between the saddle-point approach and ours stems from two things: the first one is the lack of robustness displayed in the previous test, but the second and most important reason is how we handle the unit constraint. While the monolithic approach becomes singular in ill-defined points, our smoothed projection procedure is well-defined for all vectors.

\begin{figure}[ht!]
    \centering
    \begin{tikzpicture}
        \begin{axis}[xlabel=$\theta$, ylabel=Nonlinear iters., width=0.7\textwidth, height=5cm, legend style={fill opacity=0.6, legend cell align=left, legend pos=north west}, legend columns=1]
            \addplot[red!70!black, style={ mark options={draw opacity=1.0, fill opacity=1.0, mark size=1pt}}, mark=*, line width=1.5] coordinates
            {(0.0,70)(0.1,73)(0.2,94)(0.3,69)(0.4,72)(0.5,172)(0.6,167)(0.7,163)(0.8,164)(0.9,166)(1.0,171)(1.1,177)(1.2,76)(1.3,69)(1.4,95)(1.5,73)(1.6,64)(1.7,72)(1.8,85)(1.9,111)(2.0,122)(2.1,127)(2.2,128)(2.3,129)(2.4,129)(2.5,128)(2.6,127)(2.7,123)(2.8,111)(2.9,85)(3.0,72)(3.1,64)};
            \addplot[blue!70!black, style={ mark options={draw opacity=1.0, fill opacity=1.0, mark size=1pt}}, mark=*, line width=1.5] coordinates
            {(0.1,185)(0.2,102)(0.3,149)(0.4,132)(0.5,133)(0.6,131)(0.7,127)(0.8,122)(0.9,126)(1.0,130)(1.1,134)(1.2,131)(1.3,144)(1.4,156)(1.5,302)};
            \addplot[blue!70!black, style={ mark options={draw opacity=1.0, fill opacity=1.0, mark size=1pt}}, mark=*, line width=1.5] coordinates
            {(1.7,152)(1.8,115)(1.9,90)(2.0,88)(2.1,85)(2.2,82)(2.3,79)(2.4,70)(2.5,83)(2.6,86)(2.7,89)(2.8,90)(2.9,116)(3.0,187)};
            \legend{PPGD $\mathbb P_1$, PPGD $\mathbb P_2$}
        \end{axis}
    \end{tikzpicture}
    \caption{{\bf Singularity appoximation.} Number of nonlinear iterations incurred by the preconditioned projected gradient descent (PPGD) with first and second finite elements, and the monolithic solution of the saddle point problem with discretized with $[\mathbb P_2]^k \times \mathbb P_1$ finite elements. Missing points imply that the nonlinear solver did not converge.}
    \label{fig:initial-conditions-singularity}
\end{figure}

\begin{figure}[ht!]
    \centering
    \begin{subfigure}{0.49\textwidth}
        \centering
        \includegraphics[width=0.9\textwidth]{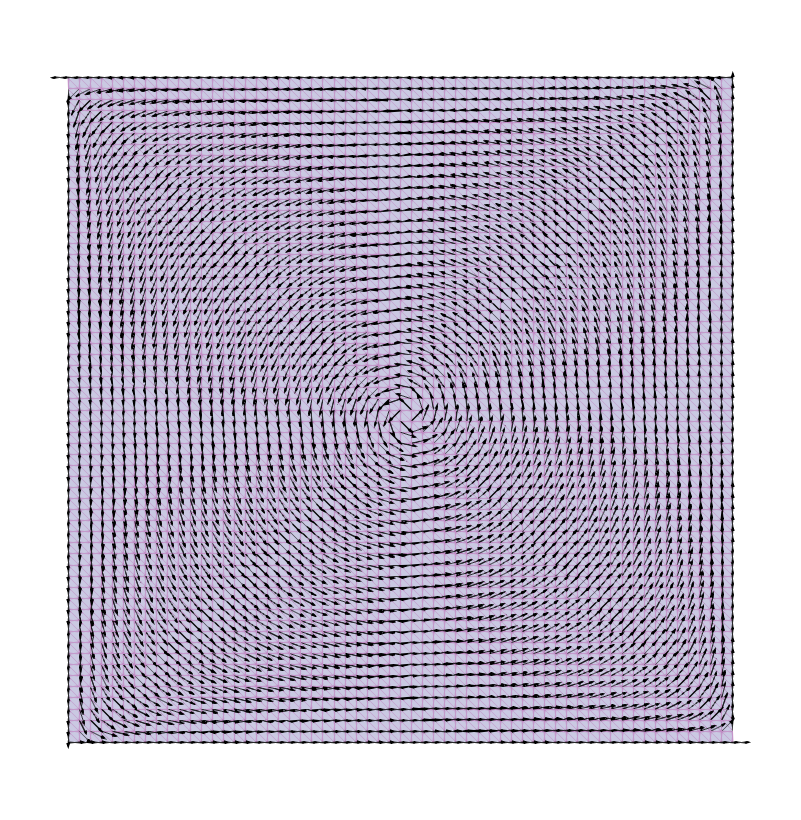}
    \end{subfigure}
    \begin{subfigure}{0.49\textwidth}
        \centering
        \includegraphics[width=0.9\textwidth]{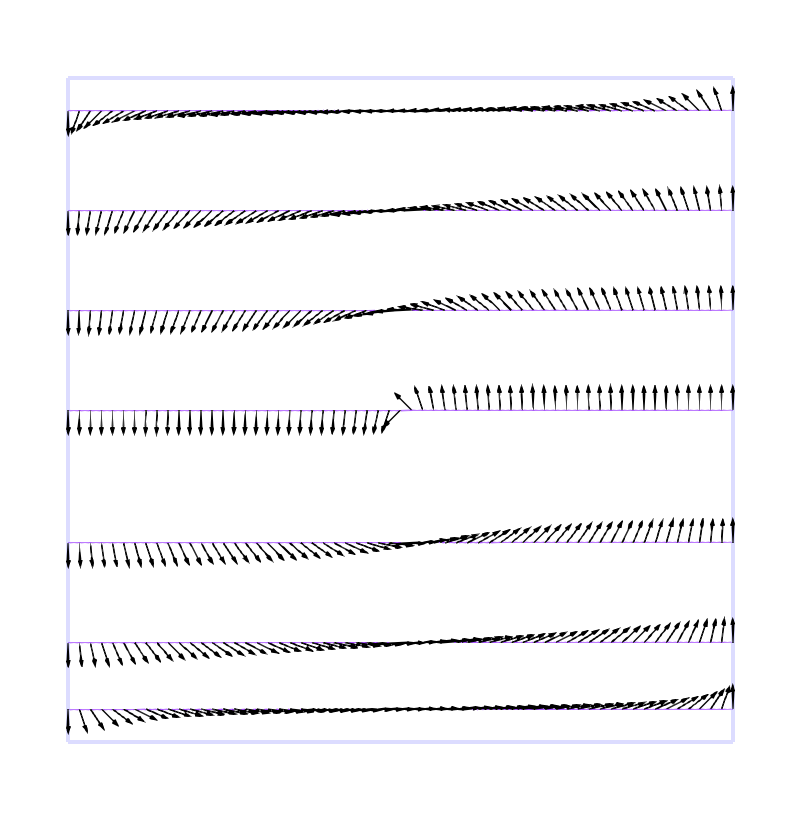}
    \end{subfigure}
    \caption{(a) Solution obtained when solving a problem that induces a singularity, computed with our proposed method. (b) Slices of lines $y=\{0.05, 0.15,0.3,0.5,0.65,0.8,0.95\}$ from the solution in (a) to display the rotation speed in bidimensional problems.}
    \label{fig:singularity}
\end{figure}

\paragraph{Optimality.} We come back to the first test, where we now perform sensitivity against the number of degrees of freedom and fix $\theta=0$. We show our results for first and second order finite elements in Figure \ref{fig:optimality}, where it can be seen how the method is optimal up to roughly more than 3 million degrees of freedom. Despite there being a small increase in the number of iterations, it is negligible as the number of degrees of freedom is growing exponentially. We note that all operations in the solver are adequate for parallelization, as there are only three types of operations required: one is residual assembly, the second one is vector operations, and the last one is the action of the preconditioner. The first one is produced (in Firedrake) by automatically generated code that is compiled upon execution for the given problem, which yields excellent peformace \cite{Pyop2AHighLRathge2012}. The second one relies on BLAS-1 operations, which are highly optimized \cite{lawson1979basic}, and the last one depends on the scalability properties of BoomerAMG, which are well-established \cite{falgout2002hypre}.

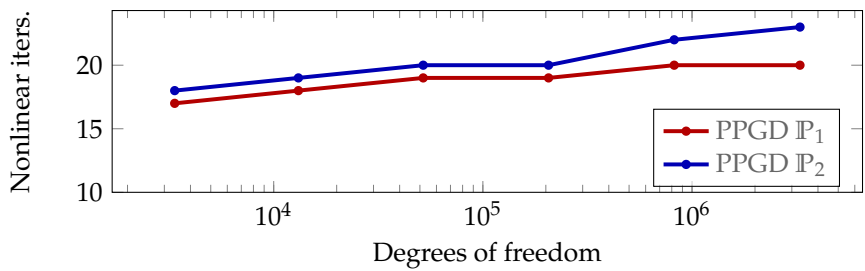
\begin{figure}
    \centering
    \begin{tikzpicture}
        \begin{semilogxaxis}[xlabel=Degrees of freedom, ylabel=Nonlinear iters., width=0.7\textwidth, height=4cm, ymin=10, legend style={fill opacity=0.6, legend cell align=left, legend pos=south east}, legend columns=1]
            \addplot[red!70!black, style={ mark options={draw opacity=1.0, fill opacity=1.0, mark size=1pt}}, mark=*, line width=1.5] coordinates
            {(3362, 17)(13122, 18)(51842, 19)(206082, 19)(821762, 20)(3281922, 20)};
            \addplot[blue!70!black, style={ mark options={draw opacity=1.0, fill opacity=1.0, mark size=1pt}}, mark=*, line width=1.5] coordinates
            {(3362, 18)(13122, 19)(51842, 20)(206082, 20)(821762, 22)(3281922, 23)};
            \legend{PPGD $\mathbb P_1$, PPGD $\mathbb P_2$}
        \end{semilogxaxis}
    \end{tikzpicture}
    \caption{Number of nonlinear iterations incurred by the preconditioned projected gradient descent (PPGD) with first and second finite elements for varying number of degrees of freedom (from 3K to 3.2M).}
    \label{fig:optimality}
\end{figure}

\section{Numerical tests}\label{section:application}
In this section we use Algorithm 1 to numerically study our model and compare it to the standard potential-based approach. For this purpose, we provide five tests. In the first one, we show how our model can be used to generate the fiber field of a left ventricle and compare the results to a potential based approach. In the second one, we compute the nematic error term $\vec S\cdot \vec f$ from hypothesis in \eqref{eq:error-nematic} to verify its validity. In the thrid one, we study the numerical convergence of the method. More specifically, we verify that an Aubin-Nitsche estimate holds, meaning that our approach converges quadratically, in contrast to the standard one that converges linearly. On the fourth test we show the impact of our model on a mechanical contraction benchmark.

\subsection{Application to left ventricle}
We will follow the steps outlined in Section \ref{section:model}. For each step, we show the analogous result with our proposed vector formulation.

\paragraph{Transmural distance and vector.} The model for the transmural potential $\phi_\Trans$ is given by \eqref{eq:poisson-scalar}, with boundary conditions
    $$ \phi_\Trans = 1 \quad\ton\,\,\Gamma_\Endo, \quad \phi_\Trans=0\quad\ton\,\,\Gamma_\Epi, \quad \grad \phi_\Trans\cdot \vec N = 0\quad\ton\,\,\Gamma_\Base.$$
In the vector model, this can be restated as finding a nematic liquid crystal vector $\vec d_\Trans$ with the boundary conditions
    $$ \vec d_\Trans = -\vec N\quad\ton\,\, \Gamma_\Endo, \quad \vec d_\Trans = \vec N\quad\ton\,\,\Gamma_\Epi,\quad \vec d_\Trans\cdot \vec N=\vec 0\quad\ton\,\,\Gamma_\Base, \quad \Pi_{\vec \tau}[\grad \vec d_\Trans]\vec N = \vec 0\quad\ton\,\,\Gamma_\Base.$$
We stress that in the vector formulation we do not compute a transmural distance as it is not necessary.
\paragraph{Apicobasal vector.} The model for the potential $\phi_\AB$ is given by the boundary conditions
    $$ \phi_\AB=1\quad\ton\,\,\Gamma_{\texttt{apex}}, \quad \phi_\AB=0\quad\ton\,\,\Gamma_\Base, \quad \grad\phi_\AB\cdot \vec N = 0\quad\text{elsewhere},$$
where $\Gamma_{\texttt{apex}}$ is a single node that represents the apex. In the vector setting, using the computations for the singularity from Section \ref{section:bcs} we obtain
    $$ \vec d_\AB = \vec N\quad\ton\,\,\Gamma_\Base,\quad \vec d_\AB = \vec 0\quad\ton\,\,\Gamma_\texttt{apex}, \quad \vec d_\AB\cdot \vec N=\vec 0\quad\ton\,\,(\Gamma_\Base\cup\Gamma_\texttt{apex})^c, \quad \Pi_{\vec \tau}[\grad \vec d]\vec N = \vec 0\quad\ton\,\,(\Gamma_\Base\cup\Gamma_\texttt{apex})^c.$$
The resulting apicobasal vector is required to be orthogonal to the transmural vector as $\vec d_\AB \leftarrow \vec d_\AB - \langle \vec d_\AB, \vec d_\Trans\rangle \vec d_\Trans$.

\paragraph{Transversal vector, fibers and cross-fibers.} The transversal vector is simply given by the cross product of the previous two $\vec d \coloneqq \vec d_\Trans \times \vec d_\AB$. In the potential framework, the fiber field is computed by  means of the transmural distance by interpolating the angles as detailed in Section \ref{section:model}, which results in $\vec f(\phi_\Trans) \coloneqq \mat Q \mat R(\phi_\Trans) \mat Q^T \vec d$. In the vector setting, we can instead impose the rotated fields as a boundary condition, so that we then solve the Frank-Oseen system under assumption \eqref{eq:error-nematic} with boundary conditions
    $$ \vec f = \mat Q \mat R(1) \mat Q^T \vec d\quad\ton\,\,\Gamma_\Endo, \quad \vec f=\mat Q\mat R(0)\mat Q^T\vec d\quad\ton\,\,\Gamma_\Epi,\quad [\grad\vec f]\vec N= \vec 0\quad\ton\,\,\Gamma_\Base.$$
We study the impact of term $\vec S\cdot \vec f$ in Section \ref{section:null-force}
We highlight that the interpolation of the boundary conditions within the tissue is handled automatically by the Frank-Oseen model. The cross-fiber direction can then be computed as $\vec d_{\texttt{cross}} = \vec d_\AB - \langle \vec d_\AB, \vec f\rangle \vec f$. The cross-fiber and transmural vectors are commonly referred to as sheet $\vec s$ and normal $\vec n$ directions.

We show the resulting fiber fields in Figure \ref{fig:vectors-lv}, together with the difference between the potential and Frank-Oseen formulations measured with the angle between both vector fields as    
    \begin{equation}\label{eq:angle-error}
        \texttt{angle} = \left|\arccos \langle \vec f_\texttt{pot}, \vec f_\texttt{FO}\rangle \right|,
    \end{equation}
where $\vec f_\texttt{pot}$ stands for the solution using the potential method, and $\vec f_\texttt{FO}$ the one using the Frank-Oseen model. The errors are very small for the transmural vector field, given by at most $14^\circ$. The apicobasal vector fields again roughly coincide, with the main difference appearing around the apex, where the singularity arises. The fiber field error presents a more complex structure, for which we have enlarged the resolution of the color bar. Despite the error being small and up to roughly $30^\circ$ within the tissue, the error is mainly concentrated around the apex, as in the apicobasal case.

\begin{figure}[ht!]
    \centering
    \def\wid{0.25}
    \begin{subfigure}[b]{\wid\textwidth}
        \includegraphics[width=\textwidth]{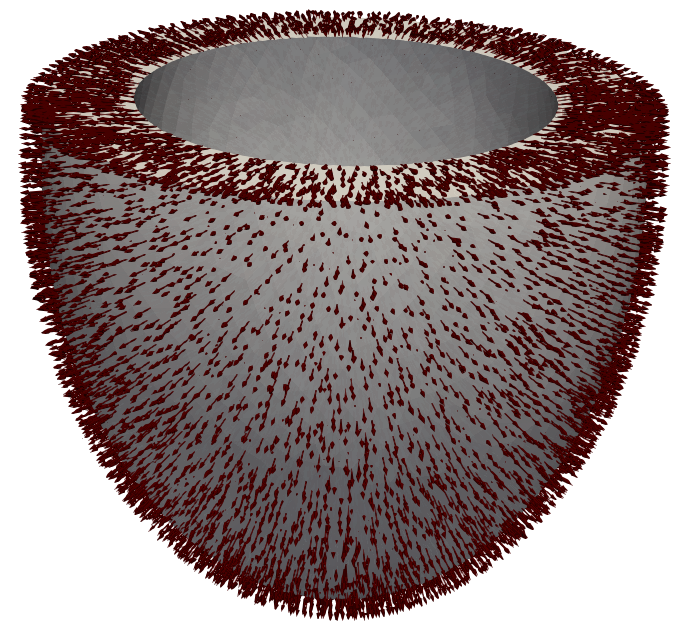}
        \caption{}
    \end{subfigure}
    \begin{subfigure}[b]{\wid\textwidth}
        \includegraphics[width=\textwidth]{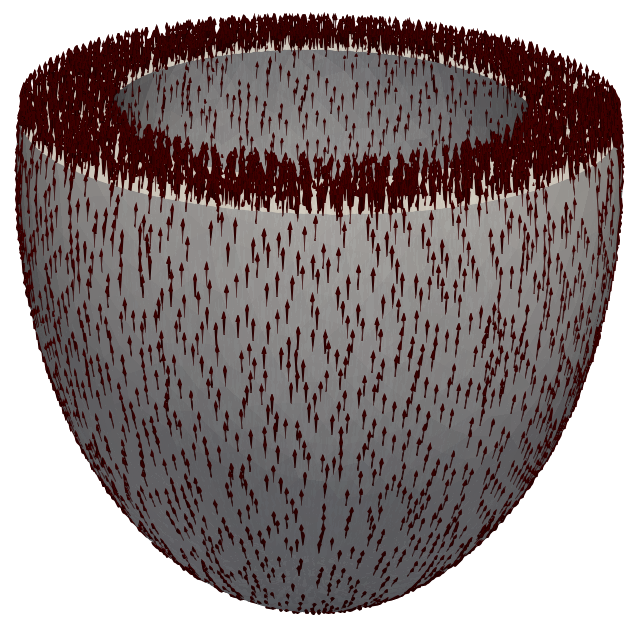}
        \caption{}
    \end{subfigure}
    \begin{subfigure}[b]{\wid\textwidth}
        \includegraphics[width=\textwidth]{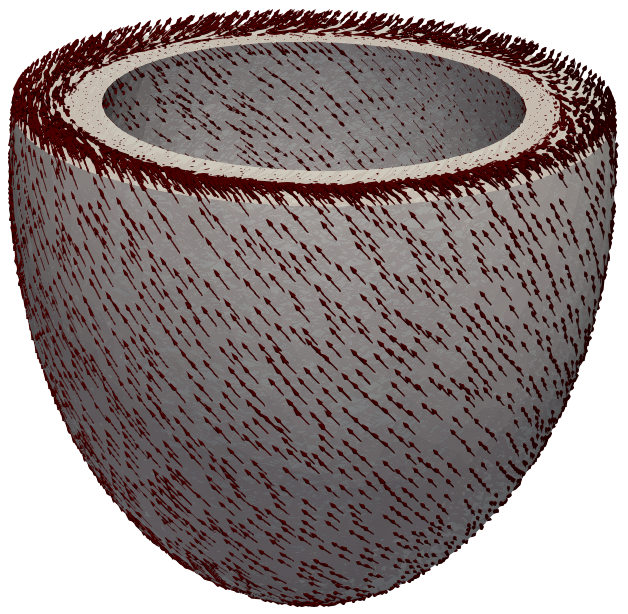}
        \caption{}
    \end{subfigure}

    \begin{subfigure}[b]{\wid\textwidth}
        \includegraphics[width=\textwidth]{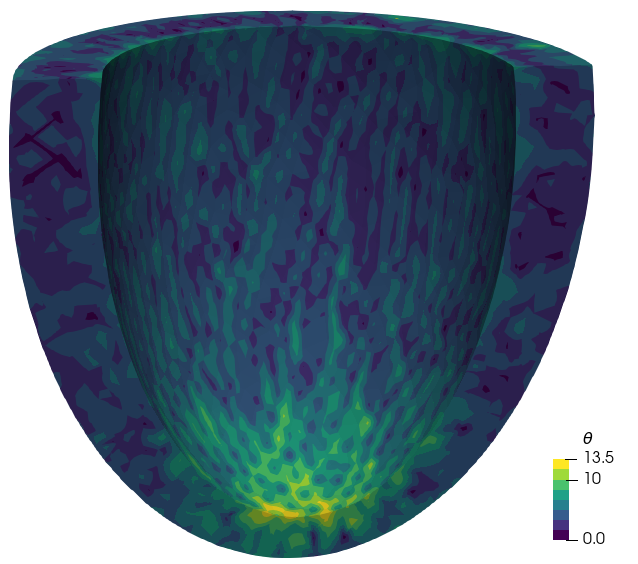}
        \caption{}
    \end{subfigure}
    \begin{subfigure}[b]{\wid\textwidth}
        \includegraphics[width=\textwidth]{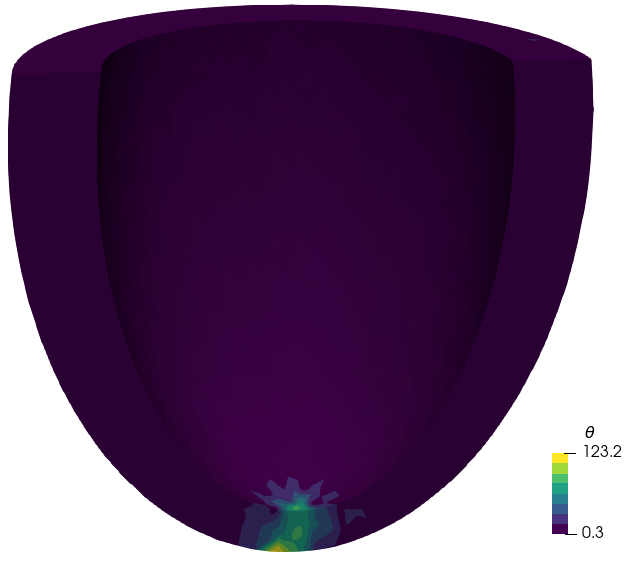}
        \caption{}
    \end{subfigure}
    \begin{subfigure}[b]{\wid\textwidth}
        \includegraphics[width=\textwidth]{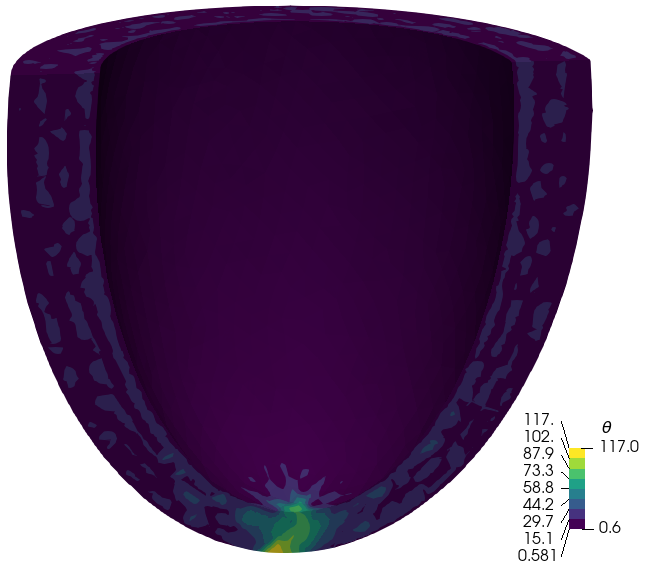}
        \caption{}
    \end{subfigure}
    \caption{Resulting vector fields in a left ventricle geometry computed as nematic liquid crystals. The errors are given by \eqref{eq:angle-error}. (a) Transmural vector field, (b) apicobasal vector field, (c) fiber vector field, (d) transmural error in degrees, (e) apicobasal error in degrees, (f) fiber error in degrees.}
    \label{fig:vectors-lv}
\end{figure}

\subsection{Impact of the forcing term $\vec S\cdot \vec f$}\label{section:null-force}
To compute the fiber field $\vec f$ without using the weight $\phi_\Trans$, we require the hypothesis $\vec S\cdot\vec f= 0$, as in \eqref{eq:error-nematic}. To explore if this is a strong or mild requirement, we have computed this quantity in all of $\Omega$ and plotted it in Figure \ref{fig:error-nematic}, where it can be seen that this quantity is comparable to the machine error. To avoid overly complicated formulas, we have computed the solution using $\mathbb P_2$ elements, so that the Laplace terms are not null by construction, which makes it surprising to see that indeed the computed term is close to machine precision. We conjecture that this term is actually zero, but a proof remains to be found.

\begin{figure}[ht!]
    \centering
    \includegraphics[width=0.4\textwidth]{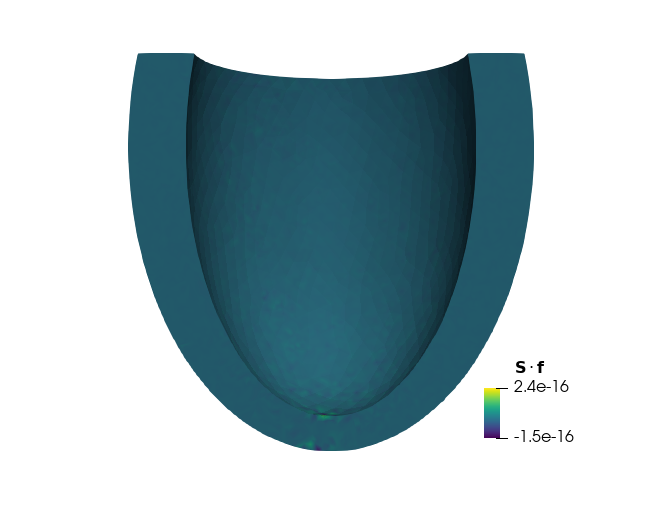}
    \caption{Forcing vector projection $\vec S\cdot \vec f$ arising from the equations governing the fiber vector field $\vec f$.}
    \label{fig:error-nematic}
\end{figure}

\subsection{Order of convergence}
One important practical advantage of our model is that it yields the relevant variable, i.e. the fibers, as a primary variable. In the RBM formulation, the cardiac fibers are computed using the gradient of a potential, which we will denote for now as $\vec d_\texttt{RBM}\coloneqq \grad \phi$. The convergence of a FEM approximation for such scheme is given in terms of the potential $\phi$, which is well-known to converge linearly \cite{quarteroni2008numerical}:
    $$ \| \vec d_\texttt{RBM}^h\|_{L^2} = \|\grad \phi^h \|_{L^2} \leq \| \phi^h \|_{H^1} \leq C h, $$
where we have denoted by $d_\texttt{RBM}^h, \phi^h$ a finite element approximation of the continuous functions $\vec d_\texttt{RBM},\phi$. This estimate still holds for the fiber field computed with our method, which we will denote by $\vec d_\texttt{FO}$:
    $$ \| \vec d_\texttt{FO}^h\|_{\vec H^1} = \leq C h, $$

 but the main difference between both approaches lies in the possibility of using the Aubin-Nitsche technique, which yields the convergence of the $L^2$ norm:
    $$ \| \vec d_\texttt{FO}^h \|_{L^2} \leq C h^2. $$
In other words, our approach yields quadratic convergence, not linear.  We have indeed verified numerically that this estimate holds, by testing the convergence of problem \eqref{eq:fo-1d} in a 2D square geometry. We show the results for up to roughly $130\,000$ degrees of freedom in Figure \ref{fig:aubin-nitsche}, where the quadratic convergence can be clearly observed.

\begin{figure}[ht!]
    \centering
    \begin{tikzpicture}
        \begin{loglogaxis}[xlabel=Mesh size, ylabel=$\vec L^2$ error, width=0.7\textwidth, height=4cm, legend style={fill opacity=0.6, legend cell align=left, legend pos=north west}, legend columns=1]
            \addplot[red!50!white, style={ mark options={draw opacity=1.0, fill opacity=1.0, mark size=1pt}}, mark=*, line width=1.5] coordinates
            { (0.3535533905932738, 0.014032474051090302)
              (0.1767766952966369, 0.0035173226930468654)
              (0.08838834764831845, 0.0008799133247377382)
              (0.04419417382415922, 0.00022001488699289378)
              (0.02209708691207961, 5.5006009375594595e-05)
              (0.011048543456039806, 1.3751645365006846e-05)
              (0.005524271728019903, 3.4379211200969684e-06) };
            \addplot[blue!50!white, style={ mark options={draw opacity=1.0, fill opacity=1.0, mark size=0pt}}, mark=*, line width=1.5] coordinates
            { (0.3535533905932738,   0.3535533905932738^2 )
              (0.1767766952966369,   0.1767766952966369^2  )
              (0.08838834764831845,  0.08838834764831845^2 )
              (0.04419417382415922,  0.04419417382415922^2 )
              (0.02209708691207961,  0.02209708691207961^2 )
              (0.011048543456039806, 0.011048543456039806^2)
              (0.005524271728019903, 0.005524271728019903^2) };
            \legend{$L^2$ error, $O(h^2)$}
        \end{loglogaxis}
    \end{tikzpicture}
    \caption{$\vec L^2$ error of a FEM approximation to the Frank-Oseen problem with respect to the mesh size.}
    \label{fig:aubin-nitsche}
\end{figure}
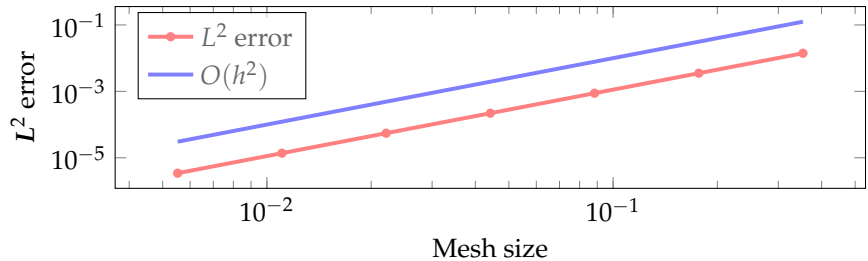

\subsection{Apex displacement comparison}
We finally studied the impact of using both fiber generation models on a cardiac mechanical simulation. Given that the point where both methods differ the most is the apex, we study the apex displacement on a benchmark contraction problem \cite{VerificationOfLand2015}. Details of the benchmark test can be found in the provided reference, and we display the deformation at $0$, $0.1$, and $0.2$ seconds at Figure \ref{fig:contraction} (a-c). We show the apex displacement magnitude in Figure \ref{fig:contraction} (d), where there is a clear difference between both methods. In particular, the displacement obtained with the Frank-Oseen model presents roughly a $25\%$ reduction in the apex displacement. This can be explained by the fact that our method is able to compute a singularity and setting the solution to $\vec 0$ in the singular point, whereas the standard approach computes an arbitrary vector at the apex, which yields an artificial active contraction force.

\begin{figure}[ht!]
    \centering
    \begin{subfigure}{0.2\textwidth}
        \centering
        \includegraphics[width=\textwidth]{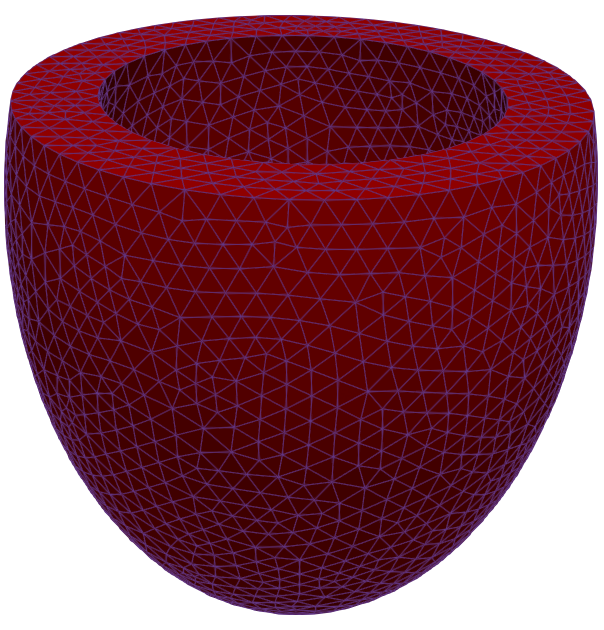}
        \caption{}
    \end{subfigure}
    \begin{subfigure}{0.2\textwidth}
        \centering
        \includegraphics[width=\textwidth]{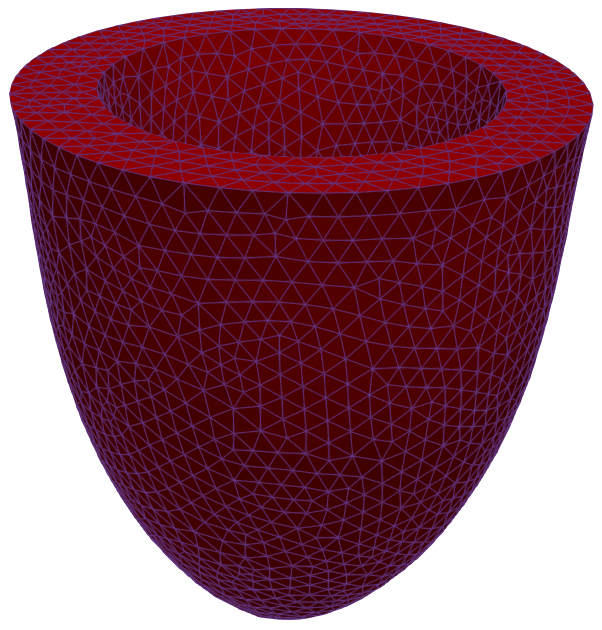}
        \caption{}
    \end{subfigure}
    \begin{subfigure}{0.2\textwidth}
        \centering
        \includegraphics[width=\textwidth]{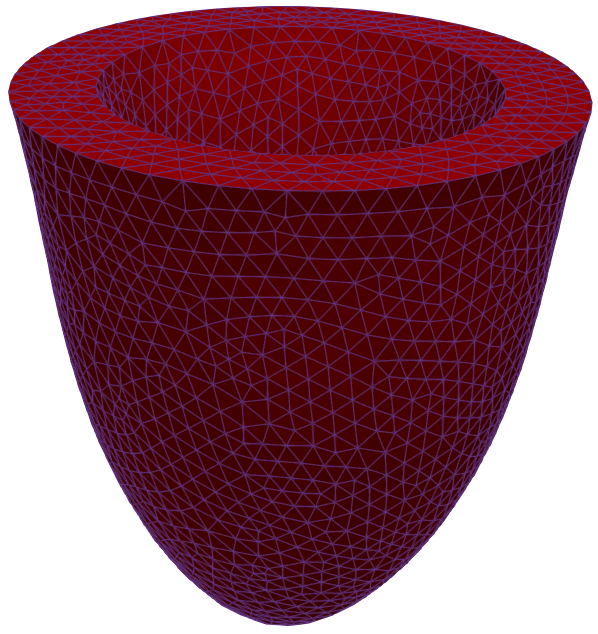}
        \caption{}
    \end{subfigure}

    \begin{subfigure}{\textwidth}
        \centering
        \includegraphics[width=\textwidth]{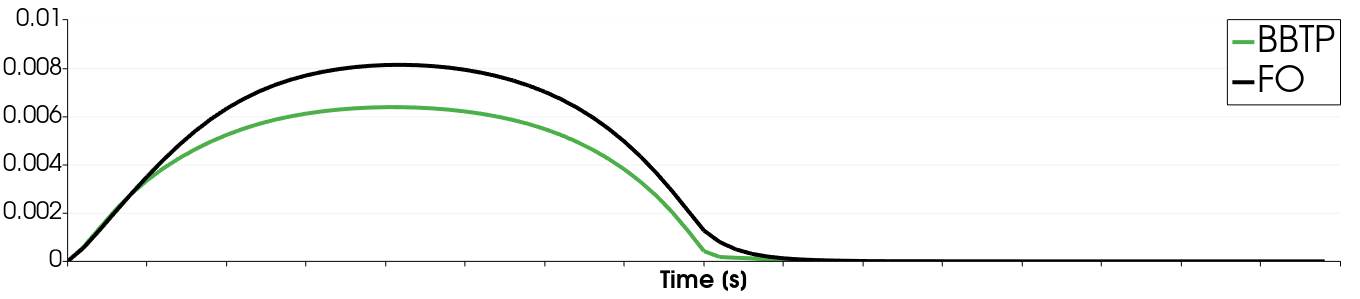}
        \caption{}
    \end{subfigure}
    \caption{Benchmark contraction with homogeneous boundary counditions at the base at (a) $t=0.0$ seconds, (b) $t=0.1$ seconds, and (c) $t=0.3$ seconds. We also show the apex displacement (d) obtained with our model (FO) and with the potential based one (BBTP).}
    \label{fig:contraction} 
\end{figure}

\section{Topological defects in cardiac fibers}\label{section:singularities}
We present some important ideas regarding the existence of singularities in the solution of the Frank-Oseen equation. All of the results referred in this section and additional details can be found in \cite{LiquidCrystalsBall2017,OrientationOfTang2017}.

Nematic liquid crystals can present a variety of singularities, understood as zones where the gradient of the crystal orientation varies unboundedly, meaning that $\lambda\to\infty$. They can be characterized with a quantity known as topological charge, and a thorough characterization of the topological defects in 2D has been done in \cite{OrientationOfTang2017}. Nevertheless, our problem \eqref{eq:fiber-min} is subject to the one-constant hypothesis, which drastically simplifies it. Two fundamental results for our case are (i) in the one-constant scenario, the set of singularities is given by a finite number of points \cite{Schoen1982regularity} and (ii) beyond the one-constant hypothesis the set of singularities has zero one-dimensional Hausdorff measure \cite{Hardt1988stable}. Note that the Hausdorff measure is an outer measure, meaning that, at least intuitively, the set of singularities in such case can be thought as an enumerable set of points at most, but not a line. 

The derivations from Section \ref{section:model} suggest that cardiac fibers behave as a nematic liquid crystal under the one-constant hypothesis, meaning that we should expect to see at most a discrete set of singularities in the solution. In \cite{TheNematicChiAuriau2022}, the authors show that by analyzing normal axis cuts of DT-MRI images it can be observed that the heart has two types of singularities. The first one is the well-known spiral orientation around the left-ventricle apex, which is typically recovered using RBMs and is indeed a pointwise defect, which we show in Figure \ref{fig:apex-singularity}. The second one is a line that originates three different main orientations, known as a disclination of index -1/2. This is observed at the junction of the ventricles and the intra-ventricular septum. Singularities of non-integer index do not arise from a direction-biased model such as Frank-Oseen and Ericksen, but instead requires a more general model such as the Landau-De Gennes. In such cases, the smoothness of the vector field in Frank-Oseen avoids the singularity by "escaping" to the third dimension.

\begin{figure}
    \centering
    \includegraphics[width=0.3\textwidth]{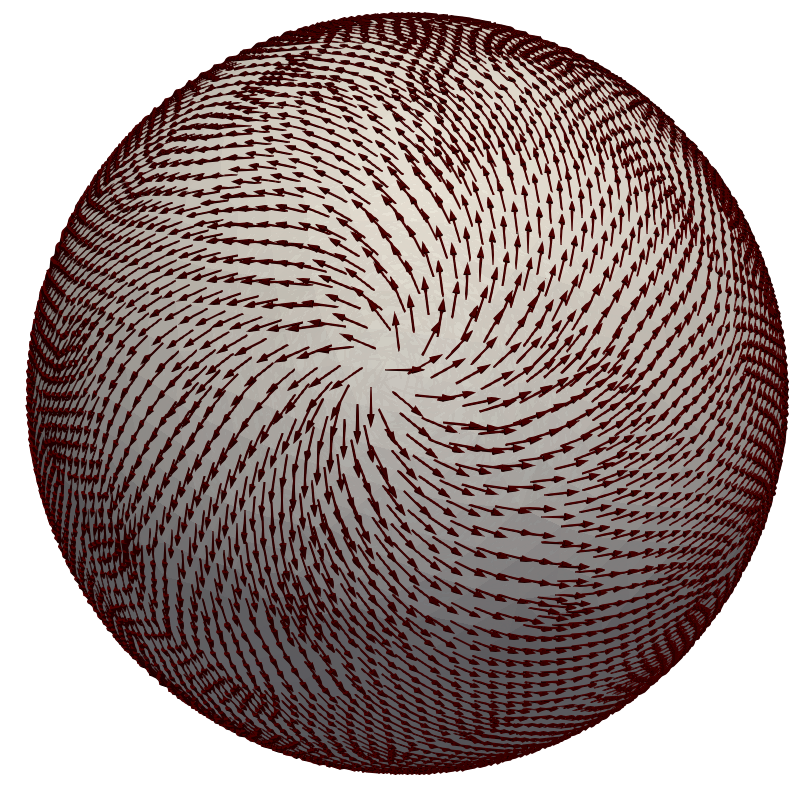}
    \caption{Spiral shaped singularity at the left ventricle apex.}
    \label{fig:apex-singularity}
\end{figure}


\section{Discussion and conclusions}\label{section:discussion}
In this work we have derived a new model for computing the fiber basis that gives rise to anisotropy in the human ventricles. The model can be seen as \emph{the} fiber PDE, which was computed starting from the RBMs and resulted in the equations used to model a nematic liquid crystal, under the one-constant hypothesis. This approach is more computationally challenging than traditional ones, as it involves the solution of a constrained nonlinear PDE, but in return it yields a fiber field that is much more accurate. We have also devised a precoditioned projected gradient descent solver that is efficient, robust and scalable for a wide range of scenarios, but its theoretical properties remain yet to be explored. This approach seems much better suited for its usage as a black-box solver than the existing ones based on the saddle point formulation.  From a physical perspective, this model can be seen as the first mathematical model for anisotropy in living tissue, as recently observed in experiments \cite{LiquidCrystalsHirst2017, TheNematicChiAuriau2022}. This opens possible future research in which other fibrous structures could be modeled within this framework, such as the skeletal muscle, the dermis, and even the DNA at a molecular level \cite{BiologicalTissSawT2018}. 

Due to the considerations exposed in Section \ref{section:singularities}, we note that this model is not useful for cardiac simulations beyond the left-ventricle case. This is due to the directionality bias that is imposed by the Frank-Oseen, which does not allow for non-integer singularities to arise. This was indeed the scope of an \emph{ad-hoc} technique developed in \cite{ANovelRuleBaBayer2012} known as \emph{bislerp}, which has the objective of providing a quaternion interpolation technique that is independent of the orientation of the vector. The natural mathematical framework for such behavior in the liquid crystals theory is using the non-directed tensor $\vec f\otimes \vec f$ instead of the vector $\vec f$, which is central to the $Q$-tensor theory in the Landau-De Gennes theory. Adapting this framework for large-scale simulations of the human heart will be the main objective of our future work.  

\section*{Acknowledgments}
NAB was supported by Centro de Modelamiento Matematico (CMM), Proyecto Basal FB210005, and by ANID, Proyecto 3230326. AO was funded by ANID-Fondecyt 1201311, 1231404, CMM FB210005 Basal-ANID, FONDAP/15110009, Math-Amsud MATH190008, DO (ANID) DO210001
and Millennium Science Initiative Programs NCN19-161 and ICN2021-004.
\bibliography{main}

\newcommand{\etalchar}[1]{$^{#1}$}
\begin{thebibliography}{RHGP{\etalchar{+}}22}

\bibitem[AEMM16]{ConstrainedOptAdler2016}
JH~Adler, DB~Emerson, SP~MacLachlan, and TA~Manteuffel.
\newblock Constrained optimization for liquid crystal equilibria.
\newblock {\em SIAM Journal on Scientific Computing}, 38:B50--B76, 1 2016.

\bibitem[AUJ22]{TheNematicChiAuriau2022}
J~Auriau, Y~Usson, and P-S Jouk.
\newblock The nematic chiral liquid crystal structure of the cardiac
  myoarchitecture: Disclinations and topological singularities.
\newblock {\em Journal of Cardiovascular Development and Disease}, 9:371, 10
  2022.

\bibitem[Bal17]{LiquidCrystalsBall2017}
JM~Ball.
\newblock {\em Liquid Crystals and Their Defects}, pages 1--46.
\newblock Springer International Publishing, 9 2017.

\bibitem[BB18]{BlockPreconditBeik2018}
FPA Beik and M~Benzi.
\newblock Block preconditioners for saddle point systems arising from liquid
  crystal directors modeling.
\newblock {\em Calcolo}, 55, 9 2018.

\bibitem[BBF13]{boffi2013mixed}
D~Boffi, F~Brezzi, and M~Fortin.
\newblock {\em Mixed finite element methods and applications}, volume~44.
\newblock Springer, 2013.

\bibitem[BBPT12]{ANovelRuleBaBayer2012}
JD~Bayer, RC~Blake, G~Plank, and NA~Trayanova.
\newblock A novel rule-based algorithm for assigning myocardial fiber
  orientation to computational heart models.
\newblock {\em Annals of Biomedical Engineering}, 40:2243--2254, 10 2012.

\bibitem[BHPS22]{barnafi2022analysis}
NA~Barnafi, NMM Huynh, LF~Pavarino, and S~Scacchi.
\newblock Analysis and numerical validation of robust parallel nonlinear
  solvers for implicit time discretizations of the bidomain equations.
\newblock {\em arXiv preprint arXiv:2209.05193}, 2022.

\bibitem[BPS22]{ParallelInexacBarnaf2022}
NA~Barnafi, LF~Pavarino, and S~Scacchi.
\newblock Parallel inexact newton\textendash{}krylov and quasi-newton solvers
  for nonlinear elasticity.
\newblock {\em Computer Methods in Applied Mechanics and Engineering},
  400:115557, 10 2022.

\bibitem[DSIB{\etalchar{+}}19]{ARuleBasedMeDoste2019}
R~Doste, D~Soto-Iglesias, G~Bernardino, A~Alcaine, R~Sebastian,
  S~Giffard-Roisin, M~Sermesant, A~Berruezo, D~Sanchez-Quintana, and O~Camara.
\newblock A rule-based method to model myocardial fiber orientation in cardiac
  biventricular geometries with outflow tracts.
\newblock {\em International Journal for Numerical Methods in Biomedical
  Engineering}, 35:e3185, 4 2019.

\bibitem[Eva22]{Evans2022partial}
LC~Evans.
\newblock {\em Partial differential equations}, volume~19.
\newblock American Mathematical Society, 2022.

\bibitem[FY02]{falgout2002hypre}
RD~Falgout and UM~Yang.
\newblock hypre: A library of high performance preconditioners.
\newblock In {\em International Conference on Computational Science}, pages
  632--641. Springer, 2002.

\bibitem[GS21]{NitschesMethoGjerde2021}
I~Gjerde and L~Scott.
\newblock Nitsche's method for navier\textendash{}stokes equations with slip
  boundary conditions.
\newblock {\em Mathematics of Computation}, 91:597--622, 11 2021.

\bibitem[HC17]{LiquidCrystalsHirst2017}
LS~Hirst and G~Charras.
\newblock Liquid crystals in living tissue.
\newblock {\em Nature}, 544:164--165, 4 2017.

\bibitem[HKL88]{Hardt1988stable}
R~Hardt, D~Kinderlehrer, and F-H Lin.
\newblock Stable defects of minimizers of constrained variational principles.
\newblock In {\em Annales de l'Institut Henri Poincar{\'e} C, Analyse non
  lin{\'e}aire}, volume~5, pages 297--322. Elsevier, 1988.

\bibitem[HTW09]{ASaddlePointHuQi2009}
Q~Hu, X-C Tai, and R~Winther.
\newblock A saddle point approach to the computation of harmonic maps.
\newblock {\em SIAM Journal on Numerical Analysis}, 47:1500--1523, 1 2009.

\bibitem[KST{\etalchar{+}}11]{ModelingAtrialKruege2011}
MW~Krueger, V~Schmidt, C~Tob\'{o}n, FM~Weber, C~Lorenz, DUJ Keller,
  H~Barschdorf, M~Burdumy, P~Neher, G~Plank, K~Rhode, G~Seemann,
  D~Sanchez-Quintana, J~Saiz, R~Razavi, and O~D\"{o}ssel.
\newblock {\em Modeling Atrial Fiber Orientation in Patient-Specific
  Geometries: A Semi-automatic Rule-Based Approach}, pages 223--232.
\newblock Springer Berlin Heidelberg, 2011.

\bibitem[LGA{\etalchar{+}}15]{VerificationOfLand2015}
S~Land, V~Gurev, S~Arens, CM~Augustin, L~Baron, R~Blake, C~Bradley, S~Castro,
  A~Crozier, M~Favino, TE~Fastl, T~Fritz, H~Gao, A~Gizzi, BE~Griffith,
  DE~Hurtado, R~Krause, X~Luo, MP~Nash, S~Pezzuto, G~Plank, S~Rossi,
  D~Ruprecht, G~Seemann, NP~Smith, J~Sundnes, JJ~Rice, N~Trayanova, D~Wang,
  Z~Jenny~Wang, and SA~Niederer.
\newblock Verification of cardiac mechanics software: benchmark problems and
  solutions for testing active and passive material behaviour.
\newblock {\em Proceedings of the Royal Society A: Mathematical, Physical and
  Engineering Sciences}, 471:20150641, 12 2015.

\bibitem[LHKK79]{lawson1979basic}
CL~Lawson, RJ~Hanson, DR~Kincaid, and FT~Krogh.
\newblock Basic linear algebra subprograms for fortran usage.
\newblock {\em ACM Transactions on Mathematical Software (TOMS)},
  5(3):308--323, 1979.

\bibitem[Lin89]{NonlinearTheorLinF1989}
F-H Lin.
\newblock Nonlinear theory of defects in nematic liquid crystals; phase
  transition and flow phenomena.
\newblock {\em Communications on Pure and Applied Mathematics}, 42:789--814, 9
  1989.

\bibitem[MZ10]{LandauDeGenneMajumd2010}
A~Majumdar and A~Zarnescu.
\newblock Landau\textendash{}{D}e {G}ennes theory of nematic liquid crystals:
  the {O}seen\textendash{}{F}rank limit and beyond.
\newblock {\em Archive for Rational Mechanics and Analysis}, 196:227--280, 4
  2010.

\bibitem[NCC19]{AShortHistoryNieder2019}
SA~Niederer, KS~Campbell, and SG~Campbell.
\newblock A short history of the development of mathematical models of cardiac
  mechanics.
\newblock {\em Journal of Molecular and Cellular Cardiology}, 127:11--19, 2
  2019.

\bibitem[PAF{\etalchar{+}}21]{ModelingCardiaPiersa2021}
R~Piersanti, PC~Africa, M~Fedele, C~Vergara, L~Ded\`{e}, AF~Corno, and
  AM~Quarteroni.
\newblock Modeling cardiac muscle fibers in ventricular and atrial
  electrophysiology simulations.
\newblock {\em Computer Methods in Applied Mechanics and Engineering},
  373:113468, 1 2021.

\bibitem[QV08]{quarteroni2008numerical}
A~Quarteroni and A~Valli.
\newblock {\em Numerical approximation of partial differential equations},
  volume~23.
\newblock Springer Science \& Business Media, 2008.

\bibitem[RHGP{\etalchar{+}}22]{PhysicsInformeRuizH2022}
C~Ruiz~Herrera, T~Grandits, G~Plank, P~Perdikaris, F~Sahli~Costabal, and
  S~Pezzuto.
\newblock Physics-informed neural networks to learn cardiac fiber orientation
  from multiple electroanatomical maps.
\newblock {\em Engineering with Computers}, 38:3957--3973, 10 2022.

\bibitem[RHM{\etalchar{+}}16]{rathgeber2016firedrake}
F~Rathgeber, DA~Ham, L~Mitchell, M~Lange, F~Luporini, ATT McRae, G-T Bercea,
  GR~Markall, and PHJ Kelly.
\newblock Firedrake: automating the finite element method by composing
  abstractions.
\newblock {\em ACM Transactions on Mathematical Software (TOMS)}, 43(3):1--27,
  2016.

\bibitem[RLRB{\etalchar{+}}14]{ThermodynamicalRossi2014}
S~Rossi, T~Lassila, R~Ruiz-Baier, A~Sequeira, and AM~Quarteroni.
\newblock Thermodynamically consistent orthotropic activation model capturing
  ventricular systolic wall thickening in cardiac electromechanics.
\newblock {\em European Journal of Mechanics - A/Solids}, 48:129--142, 11 2014.

\bibitem[RMM{\etalchar{+}}12]{Pyop2AHighLRathge2012}
F~Rathgeber, GR~Markall, L~Mitchell, N~Loriant, DA~Ham, C~Bertolli, and PHJ
  Kelly.
\newblock Pyop2: A high-level framework for performance-portable simulations on
  unstructured meshes.
\newblock In {\em 2012 SC Companion: High Performance Computing, Networking,
  Storage and Analysis (SCC)}. IEEE, 11 2012.

\bibitem[RS82]{EffectOfTissuRobert1982}
DE~Roberts and AM~Scher.
\newblock Effect of tissue anisotropy on extracellular potential fields in
  canine myocardium in situ.
\newblock {\em Circulation Research}, 50:342--351, 3 1982.

\bibitem[RWS{\etalchar{+}}19]{ATechniqueForRoney2019}
CH~Roney, J~Whitaker, I~Sim, L~O'Neill, RK~Mukherjee, O~Razeghi, EJ~Vigmond,
  M~Wright, MD~O'Neill, SE~Williams, and SA~Niederer.
\newblock A technique for measuring anisotropy in atrial conduction to estimate
  conduction velocity and atrial fibre direction.
\newblock {\em Computers in Biology and Medicine}, 104:278--290, 1 2019.

\bibitem[Sch99]{MultigridMethoSchobe1999}
J~Sch\"{o}berl.
\newblock Multigrid methods for a parameter dependent problem in primal
  variables.
\newblock {\em Numerische Mathematik}, 84:97--119, 11 1999.

\bibitem[Sho85]{AnimatingRotatShoema1985}
K~Shoemake.
\newblock Animating rotation with quaternion curves.
\newblock In {\em the 12th annual conference}. ACM Press, 1985.

\bibitem[SSP{\etalchar{+}}69]{FiberOrientatiStreet1969}
DD~Streeter, HM~Spotnitz, DP~Patel, J~Ross, and EH~Sonnenblick.
\newblock Fiber orientation in the canine left ventricle during diastole and
  systole.
\newblock {\em Circulation Research}, 24:339--347, 3 1969.

\bibitem[Ste95]{stenberg1995some}
R~Stenberg.
\newblock On some techniques for approximating boundary conditions in the
  finite element method.
\newblock {\em Journal of Computational and applied Mathematics},
  63(1-3):139--148, 1995.

\bibitem[SU82]{Schoen1982regularity}
R~Schoen and K~Uhlenbeck.
\newblock A regularity theory for harmonic maps.
\newblock {\em Journal of Differential Geometry}, 17(2):307--335, 1982.

\bibitem[SXLL18]{BiologicalTissSawT2018}
TB~Saw, W~Xi, B~Ladoux, and CT~Lim.
\newblock Biological tissues as active nematic liquid crystals.
\newblock {\em Advanced Materials}, 30:1802579, 11 2018.

\bibitem[TD18]{tortora2018principles}
GJ~Tortora and BH~Derrickson.
\newblock {\em Principles of anatomy and physiology}.
\newblock John Wiley \& Sons, 2018.

\bibitem[TS17]{OrientationOfTang2017}
X~Tang and JV~Selinger.
\newblock Orientation of topological defects in 2{D} nematic liquid crystals.
\newblock {\em Soft Matter}, 13:5481--5490, 2017.

\bibitem[Vr86]{FiniteElementVerfr1986}
R~Verf\"u~rth.
\newblock Finite element approximation on incompressible navier-stokes
  equations with slip boundary condition.
\newblock {\em Numerische Mathematik}, 50:697--721, 11 1986.

\bibitem[WK14]{GeneratingFibrWong2014}
J~Wong and E~Kuhl.
\newblock Generating fibre orientation maps in human heart models using poisson
  interpolation.
\newblock {\em Computer Methods in Biomechanics and Biomedical Engineering},
  17:1217--1226, 8 2014.

\bibitem[WN99]{wright1999numerical}
S~Wright and J~Nocedal.
\newblock Numerical optimization.
\newblock {\em Springer Science}, 35(67-68):7, 1999.

\bibitem[XFW21]{AugmentedLagraXiaJ2021}
J~Xia, PE~Farrell, and F~Wechsung.
\newblock Augmented lagrangian preconditioners for the
  {O}seen\textendash{}{F}rank model of nematic and cholesteric liquid crystals.
\newblock {\em BIT Numerical Mathematics}, 61:607--644, 6 2021.

\bibitem[ZAS{\etalchar{+}}21]{AnAutomatePipZheng2021}
T~Zheng, L~Azzolin, J~S\'{a}nchez, O~D\"{o}ssel, and A~Loewe.
\newblock An automate pipeline for generating fiber orientation and region
  annotation in patient specific atrial models.
\newblock {\em Current Directions in Biomedical Engineering}, 7:136--139, 10
  2021.

\bibitem[ZBne]{HeartMuscleFiZhukovNone}
L~Zhukov and AH~Barr.
\newblock Heart-muscle fiber reconstruction from diffusion tensor {MRI}.
\newblock In {\em IEEE Visualization 2003}. IEEE, None.

\end{thebibliography}
\bibliographystyle{alpha}

\end{document}